\documentclass[a4paper]{amsart}
\usepackage{amsthm}
\usepackage{amssymb}
\usepackage{amsmath}
\usepackage{verbatim,enumerate}

\renewcommand{\thefootnote}{\fnsymbol{footnote}}

\newtheorem{theorem}{Theorem}[section]
\newtheorem{lemma}[theorem]{Lemma}
\newtheorem{proposition}[theorem]{Proposition}
\newtheorem{remark}[theorem]{Remark}

\newtheorem{example}[theorem]{Example}

\newtheorem*{example*}{Example}

\newtheorem*{theorem*}{Theorem}

\newtheorem*{remark*}{Remark}

\newtheorem{corollary}[theorem]{Corollary}
\newtheorem*{corollary*}{Corollary}

\newtheorem{definition}[theorem]{Definition}

\newtheorem*{definition*}{Definition}

\newtheorem*{notation*}{Notation}

\newtheorem{notation}[theorem]{Notation}

\numberwithin{equation}{section}

\gdef\myletter{}

\let\savetheequation\theequation
\def\theequation{\savetheequation\myletter}

\def\calB{\mathcal{B}}
\def\calS{\mathcal{S}}
\def\calC{\mathcal{C}}
\def\calF{\mathcal{F}}
\def\bW{\mathbf{W}}   \def\bV{\mathbf{V}}
\def\lto{\longrightarrow} 

\newcommand{\CC}{{\mathbb C}}

\newcommand{\RR}{{\mathbb R}}
\newcommand{\ZZ}{{\mathbb Z}}

\newcommand{\PP}{{\mathbb P}}

\newcommand{\calM}{\mathcal{M}}
\newcommand{\lt}{\textsc{lt}}

\newcommand{\vol}{\mathrm{vol}}

\newcommand{\Van}{{\mathrm{Van}}}

\def \bV{\mathbf{V}}
\def \bI{\mathbf{I}}
\def \bar{\overline}
\def \hat{\widehat}

\def \bv{{\bf v}}

\def \b0{{\bf 0}}
\def \be{\mathbf{e}}

\def \span{\mathrm{span}}

\def \LT{\hbox{\textsc{lt}}}

\def \hsta{\,\hat*\,}

\long\def\symbolfootnote[#1]#2{\begingroup%
\def\thefootnote{\fnsymbol{footnote}}\footnote[#1]{#2}\endgroup}

\begin{document}

\title{Transfinite diameter on complex algebraic varieties}

\author{David A. Cox}
\address{Department of Mathematics \& Statistics, 
Amherst College, Amherst, MA 01002, USA} 
\email{dacox@amherst.edu}

\author{Sione Ma`u}
\address{Department of Mathematics,
University of Auckland,
Auckland, NZ}
\email{s.mau@auckland.ac.nz}

\begin{abstract} 
We use methods from computational algebraic geometry to study 
Chebyshev constants and the transfinite diameter of a pure
$m$-dimensional affine algebraic variety in $\CC^n$ ($m\leq n$).  The
main result is a generalization of Zaharjuta's integral formula for
the Fekete-Leja transfinite diameter.
\end{abstract}

\keywords{Chebyshev constant, transfinite diameter, Vandermonde
  determinant, affine variety, Noether normalization, monomial order}
\subjclass[2010]{32U20; 14Q15}

\maketitle

\maketitle

\section{Introduction}

This paper studies a notion of transfinite diameter on a pure
$m$-dimensional algebraic subvariety of $\CC^n$, $1\leq m\leq n$.
This is a natural generalization of the Fekete-Leja transfinite
diameter in $\CC^n$, which is an important quantity in pluripotential
theory and polynomial approximation.  In the study of the Fekete-Leja
transfinite diameter in $\CC^n$ ($n>1$), an important paper is that of
Zaharjuta \cite{zaharjuta:transfinite}.  Given a compact set
$K\subseteq\CC^n$, Zaharjuta showed that its Fekete-Leja transfinite
diameter, denoted $d(K)$, was given by a well-defined limiting process
analogous to the one-dimensional case.  The main result of
\cite{zaharjuta:transfinite} is an integral formula that realizes
$d(K)$ as a ``geometric average'' of so-called \emph{directional
  Chebyshev constants} associated to $K$; these constants measure (in
an asymptotic sense) the minimum size on $K$ of polynomials with
prescribed leading terms.

Further developments and generalizations make use of the essential
techniques in \cite{zaharjuta:transfinite}.  In
\cite{jedrzejowski:homogeneous} the notion of \emph{homogeneous}
transfinite diameter was studied and a Zaharjuta-type formula proved.
In \cite{laurumely:arithmetic}, and later in \cite{rumelylauvarley:existence}, Lau, Rumely and Varley developed Zaharjuta's
techniques in the setting of arithmetic geometry to study the notion
of \emph{sectional capacity}.  More recently, Bloom and Levenberg
studied a notion of \emph{weighted} transfinite diameter in $\CC^n$
(\cite{bloomlev:weighted}, \cite{bloomlev:transfinite}).

In \cite{baleikorocaumau:chebyshev} a notion of transfinite diameter
was defined and studied on an algebraic curve $V\subseteq\CC^n$.  It was
shown that Zaharjuta's arguments, which exploit standard algebraic
properties of polynomials, may be adapted to handle algebraic
computations in the coordinate ring of $V$.  Well-developed methods
exist to carry out such computations, using Groebner bases.  In this
paper we will apply these methods to higher dimensional algebraic
varieties.

We should mention here that the notion of transfinite diameter on algebraic varieties may be studied as a by-product of 
Berman and Boucksom's general theory of Monge-Amp\`ere energy on compact complex manifolds  \cite{bermanboucksom:growth}.  Their methods are quite
different to those of this paper.

Before we describe the contents of the paper more specifically, we
briefly recall the definition of the Fekete-Leja transfinite diameter.   

Let $\{z^{\alpha_j}\}_{j=1}^{\infty}$ be the monomials in $n$
variables listed according to a \emph{graded order} (i.e.,
$|\alpha_j|\leq|\alpha_k|$ whenever $j<k$).  Here we are using
standard multi-index notation: if $\alpha_j =
(\alpha_{j1},\dots,\alpha_{jn}) \subseteq \ZZ_{\geq 0}^n$, then
$z^{\alpha_j}=z_1^{\alpha_{j1}}z_2^{\alpha_{j2}}\cdots
z_n^{\alpha_{jn}}$ and $|\alpha_j|=\alpha_{j1}+\cdots+\alpha_{jn}$
denotes the total degree.  Write $\be_j=z^{\alpha_j}$; so for
$a=(a_1,\dots,a_n)\in\CC^n$ we have $\be_j(a) =
a_1^{\alpha_{j1}}\cdots a_n^{\alpha_{jn}}$.  Given a positive integer
$M$ and points $\{\zeta_1,\dots,\zeta_M\}\subseteq\CC^n$, the $M\times
M$ determinant
\begin{equation}
\label{eqn:0.1}
\Van(\zeta_1,\dots,\zeta_M) \ = \ \det
\bigl(\be_j(\zeta_i)\bigr)_{i,j=1}^M  \ = \  
\det
\begin{pmatrix}
1 & 1 & \cdots & 1 \\
\be_2(\zeta_1) & \be_2(\zeta_2) & \cdots & \be_2(\zeta_M) \\
\vdots & \vdots & \ddots & \vdots \\
\be_M(\zeta_1) & \be_M(\zeta_2) & \cdots & \be_M(\zeta_M)
\end{pmatrix}
\end{equation}
is called a \emph{Vandermonde determinant} of order $M$. (Note that
$\be_1=1$.) 

Let $K\subseteq\CC^n$ be compact and $s$ a positive integer. Let $m_s$
be the number of monomials of degree at most $s$ in $n$ variables, and
let $l_s=\sum_{j=1}^{m_s} |\alpha_j|$ be the sum of the
degrees. Define the $s$-th order diameter of $K$ by
\begin{equation}\label{eqn:0.2}
d_s(K) \ := \ \sup\{|\Van(\zeta_1,\dots,\zeta_{m_s})|^{\frac{1}{l_s}}:
\{\zeta_1,\dots,\zeta_{m_s}\}\subseteq K\}. 
\end{equation}
The Fekete-Leja transfinite diameter of $K$ is defined as 
$\displaystyle d(K):=\limsup_{s\to\infty}d_s(K)$.

In this paper, we construct a basis $\calC$ of polynomials for the
coordinate ring $\CC[V]$ of a pure $m$-dimensional algebraic variety
$V\subseteq\CC^n$ ($1\leq m\leq n$) of degree $d$, as long as the ring
satisfies certain algebraic conditions (see \eqref{eqn:properties}).
Write $\calC=\{\be_j\}_{j=1}^{\infty}$ for this basis which we assume
is listed in a graded ordering: $\deg(\be_j)\leq\deg(\be_k)$ if $j<k$.
We define $\Van_{\calC}(\zeta_1,\dots,\zeta_M)$ to be the Vandermonde
determinant with respect to $\calC$ using the formula \eqref{eqn:0.1}.

Define $m_s=m_s(V)$ to be the number of elements of $\calC$ of degree
at most $s$, and let $l_s=l_s(V)=\sum_{j=1}^{m_s} \deg(\be_j)$ be the
sum of the degrees.  The $s$-th order diameter of a compact set
$K\subseteq V$ is defined as in \eqref{eqn:0.2} with
$\Van_{\calC}(\cdot)$ replacing $\Van(\cdot)$ on the right-hand side.
Our main theorem (Theorem \ref{thm:5.2a}) says the following.

\begin{theorem*}
The limit $d(K):=\lim_{s\to\infty}d_s(K)$ exists and 
\[
d(K) = \Bigl(\prod_{j=1}^{d} T(K,\lambda_j)\Bigr)^{\frac{1}{d}}.
\]
\end{theorem*}

Following Zaharjuta's terminology, the quantities $T(K,\lambda_j)$ on
the right-hand side are called \emph{principal Chebyshev constants}
and are defined in Section \ref{sec:cheby} as integral averages of
so-called \emph{directional Chebyshev constants}. Here $d$ is the
\emph{degree} of $V$ and the $\lambda_j$'s are the $d$ points of
intersection of the projective closure of $V$ in $\PP^n$ with a
certain subspace of the hyperplane at infinity.  When $V$ is a curve
the above result is in \cite{baleikorocaumau:chebyshev}.\footnote{The
  principal Chebyshev constants in this paper are called directional
  Chebyshev constants in \cite{baleikorocaumau:chebyshev}; for a
  one-dimensional curve, the $\lambda_j$'s may be interpreted as the
  directions of its linear asymptotes.}  When $\deg(V)=1$ then there
is only one principal Chebyshev constant, and one recovers Zaharjuta's
formula, up to a normalization.

In Section \ref{sec:2} we give some of the background needed for
subsequent sections, including Noether normalization, the grevlex
monomial ordering, normal forms and Hilbert functions.

In Section \ref{sec:3} we construct a basis (denoted by $\calC$) of
polynomials on the variety.  The basis $\calC$ consists of $d$ groups
of polynomials associated to the Noether normalization (elements of
the form $(**)$, see Proposition \ref{prop:span}), together with a
``smaller'' collection of monomials (elements of the form $(*)$).
When $V$ is a hypersurface, the basis $\calC$ can be computed rather
explicitly.

Section \ref{sec:4} is a general study of \emph{weakly
  submultiplicative functions}.  In \cite{bloomlev:weighted} it was
observed that Zaharjuta's computations with polynomials can be
reformulated abstractly as properties of submultiplicative functions.
We verify here that the relevant calculations go through with small
modifications under slightly weaker conditions.

In Section \ref{sec:cheby}, directional and principal Chebyshev
constants are defined and studied.  The main point is to construct
weakly submultiplicative functions using computational properties of
the basis $\calC$ (Corollary \ref{cor:4.4}).  The results of Sections
\ref{sec:3} and \ref{sec:4} can then be applied to this setting.

In Section \ref{sec:6} we prove the main theorem relating transfinite diameter to Chebyshev constants.  The standard
argument, based on estimating ratios of Vandermonde determinants with
directional Chebyshev constants, goes through in its entirety. 

In Section \ref{sec:7}, we show in Theorem \ref{thm:68} that the transfinite diameter may be computed using the standard basis of monomials on the variety (i.e., those monomials that give normal forms). This uses the fact that, up to a geometric factor in some finite set---the collection of $\bv_i$'s in Proposition \ref{prop:span}---each polynomial in the basis $\calC$ is a monomial.

In the appendix we compare our method to that of Rumely, Lau and Varley 
\cite{rumelylauvarley:existence}, whose so-called \emph{monic basis} is constructed by generating basis elements multiplicatively from a finite collection of polynomials with prescribed behaviour.  We compare both methods concretely in the case of the sphere in $\CC^3$.

\section{Background material} 
\label{sec:2}

We begin with Noether normalization.  Consider an ideal
$I\subseteq\CC[z_1,\dots,z_n]$ with the following properties:
\begin{enumerate}
\item \label{item:1} $\CC[z_1,\dots,z_m]\cap I=\{0\}$; and
\item \label{item:2} For each $i=m+1,\dots,n$ there exists a $g_i\in
  I$ which can be written in the form 
\begin{equation}
\label{eqn:noether}
g_i\ = \ z_i^{d_i} + \sum_{j=0}^{d_i-1}h_{ij}(z_1,\dots,z_{i-1})z_i^j,
\quad \hbox{with } \deg(h_{ij})+j\leq d_i \hbox{ for all } i. 
\end{equation}
\end{enumerate}
Property \eqref{item:1} is equivalent to saying that the map
$\CC[z_1,\dots,z_m]\to\CC[z_1,\dots,z_n]/I$, induced by the inclusion
into $\CC[z_1,\dots,z_n]$, is injective, and property \eqref{item:2}
implies that the quotient is finite over $\CC[z_1,\dots,z_m]$.  
The Noether normalization theorem says that one can always make a change of variables so that the above properties hold.  We state a specialized version of this theorem (cf.\ \cite{greuelpfister:singular}, Theorem 3.4.1).

\begin{theorem}[Noether Normalization]
Let $J\subseteq\CC[x_1,\dots,x_n]$ be an ideal.  Then there is a
positive integer $m\leq n$ and a complex linear change of coordinates
$z=T(x)$, $z_i=\sum_{j=1}^n T_{ij}x_j$, such that the following
properties hold (write $I=T(J)$):
\begin{enumerate}
\item[{\rm(1)}] The map $\CC[z_1,\dots,z_m]\to \CC[z_1,\dots,z_n]/I$ induced by
  inclusion is injective, and exhibits $\CC[z_1,\dots,z_n]/I$ as a
  finite $\CC$-algebra over $\CC[z_1,\dots,z_m]$.
\item[{\rm(2)}] For $i=m+1,\dots,n$, we can find polynomials $g_i\in I$ that
  satisfy \eqref{eqn:noether}.\qed
\end{enumerate} 
\end{theorem}

When property \eqref{item:1} of the theorem holds, we write $\CC[z_1,\dots,z_m]
\subseteq \CC[z_1,\dots,z_n]/I$.  This inclusion is called a
\emph{Noether normalization}.  All Noether normalizations used in this
paper will be assumed to satisfy the additional condition
\eqref{item:2} of the theorem since the degree condition in \eqref{eqn:noether} will be important.

The grevlex ordering, which we will denote here by $<_{gr}$, is the
ordering defined on $\ZZ_{\geq 0}^n$ by $\alpha<_{gr}\beta$ if:
\begin{enumerate}
\item $|\alpha|<|\beta|$; or, 
\item $|\alpha|=|\beta|$, and for some $i\in\{1,\dots,n\}$ we have
  $\alpha_i<\beta_i$ and $\alpha_j=\beta_j,\ \forall\,j<i$. 
\end{enumerate}
Define grevlex on monomials by putting $z^{\alpha}<_{gr}z^{\beta}$ if
$\alpha<_{gr}\beta$. More precisely, this gives the grevlex ordering
with $z_1<_{gr}z_2<_{gr}\cdots<_{gr}z_n$.  Note that $|\alpha| <
|\beta|$ implies $z^{\alpha}<_{gr}z^{\beta}$. A monomial ordering that
satisfies this property is called a \emph{graded ordering}.

Denote by $\LT(p)$ the leading term of a polynomial with respect to
grevlex, and for an ideal $I$ put $\LT(I):=\{\LT(p):p\in I\}$.  It is
well-known that for each element of $\CC[z_1,\dots,z_n]/I$ there is a
unique polynomial representative, the \emph{normal form (with respect to grevlex)}, which
contains no monomials in the ideal $\langle\LT(I)\rangle$.  If an
element of $\CC[z_1,\dots,z_n]/I$ contains the polynomial $p$, then
the normal form $r$ may computed in practice as the remainder on
dividing $p$ by a Groebner basis of $I$ (cf.\
\cite{coxlittleoshea:ideals}, \S5.3).  

Write $\CC[z]_I=\CC[z_1,\dots,z_n]_I$ for the collection of normal
forms of elements of $\CC[z]/I = \CC[z_1,\dots,z_n]/I$.  As a vector
space, $\CC[z]_I$ has a basis consisting of all monomials $z^\gamma
\notin \langle \LT(I)\rangle$.  We can give $\CC[z]_I$ the structure
of an algebra over $\CC$ with multiplication operation given by
\[
(r_1,r_2)\longmapsto \hbox{\sl ``the normal form of $r_1r_2$''}.
\]
We will usually denote this by $r_1r_2$, though we will write 
$r_1{*}r_2$ when we want to emphasize that this is the normal form of the
ordinary product.  Note that
$\CC[z]_I$ and $\CC[z]/I$ are isomorphic as $\CC$-algebras, where the
isomorphism is given by identifying normal forms with their polynomial
classes.  

Hilbert functions play an important role in some of our proofs.  We
begin with $\CC[z]_{\le s} = \CC[z_1,\dots,z_n]_{\le s}$, which
consists of polynomials of degree $\le s$.  Recall that
\begin{equation}
\label{eqn:HFCz}
\dim\,\CC[z_1,\dots,z_n]_{\leq s}=\binom{s+n}{n}=
\frac{(s+n)\cdots(s+1)}{n!} = 
\frac{1}{n!}s^n + O(s^{n-1}).
\end{equation}
Then $(\CC[z]/I)_{\le s}$ consists of all classes represented by a
polynomial of degree $\le s$.  The dimension $\dim\, (\CC[z]/I)_{\le
  s}$ gives the Hilbert function of $I$.  We also define $\CC[z]_{I\le s}$
to consist of all normal forms of degree $\le s$.  Since $<_{gr}$ is a
graded order, the isomorphism $\CC[z]_I \simeq \CC[z]/I$ induces an
isomorphism
\[
\CC[z]_{I\le s} \simeq (\CC[z]/I)_{\le s}
\]
(see \cite{coxlittleoshea:ideals}, \S9.3).  This has two useful
consequences: 
\begin{itemize}
\item The Hilbert function $\dim (\CC[z]/I)_{\le s}$ is given by the
  number of monomials $z^\gamma \notin \langle\LT(I)\rangle$ of degree
  $\le s$.
\item If $r_1 \in \CC[z]_{I\le s}$ and $r_2 \in \CC[z]_{I\le t}$, then
  $r_1{*}r_2 \in \CC[z]_{I\le s+t}$.
\end{itemize}

A Noether normalization $\CC[z_1,\dots,z_m]\subseteq\CC[z]/I$ has the
following properties.

\begin{proposition} 
\label{prop:2.2a}
Every element of $\CC[z_1,\dots,z_m]$ is a normal form, so that
\[
\CC[z_1,\dots,z_m] \subseteq \CC[z]_I.
\]
Furthermore, for $i=m+1,\dots,n$, we have
$z_i^{d_i}\in\langle\LT(I)\rangle$, where $d_i$ is as in
\eqref{eqn:noether}.
\end{proposition}

\begin{proof}
For the second assertion of the proposition, suppose
$i\in\{m+1,\dots,n\}$ and $g_i\in I$ is as in \eqref{eqn:noether}.
Then the definition of grevlex and the degree condition in
\eqref{eqn:noether} makes it  easy to see that $\LT(g_i)=z_i^{d_i}$,
which implies $z_i^{d_i}\in\langle\LT(I)\rangle$.

Since normal forms are known to form a subspace, it suffices to show
that every monomial in $\CC[z_1,\dots,z_m]$ is a normal form.  Let
$\alpha=(\alpha_1,\alpha_2,\dots,\alpha_m,0,\dots,0)$, so that
$z^{\alpha}=z_1^{\alpha_1}z_2^{\alpha_2}\cdots z_m^{\alpha_m}$.  We
want to show that $z^{\alpha}\not\in\langle\LT(I)\rangle$.

Suppose not, i.e., $z^{\alpha}\in\langle\LT(I)\rangle$.  We will
obtain a contradiction by studying the Hilbert function.  Take
$z^{\gamma}\notin\langle\LT(I)\rangle$, where
$\gamma=(\gamma_1,\dots,\gamma_n)$.  
If $i\geq m+1$, then $z_i^{d_i} \in \langle\LT(I)\rangle$, so $z_i^{d_i}$ cannot divide
$z^{\gamma}$. Hence  
\begin{equation}
\label{eqn:2.2A}
\gamma_i < d_i, \quad \text{for \emph{all} } i = m+1,\dots,n.
\end{equation} 
Furthermore, $z^\alpha\in\langle\LT(I)\rangle$, so $z^\alpha$
cannot divide $z^\gamma$. Then
\begin{equation}
\label{eqn:2.2B}
\gamma_i < \alpha_i, \quad \text{for \emph{some} } i=1,\dots,m.
\end{equation}
 
Now let
\[
L(s) \ := \ \{\gamma:
z^{\gamma}\notin\langle\LT(I)\rangle,\ |\gamma|\leq s\},  
\]
so that $|L(s)| = \dim\, (\CC[z]/I)_{\le s}$ is the Hilbert function.
Also, for $i =1,\dots,m$, let
\[
L_i(s) = \{\gamma \in L(s) : \gamma_i < \alpha_i \text{ and
}\gamma_{m+1} < d_{m+1},\dots,\gamma_n < d_n \}.
\]
Then \eqref{eqn:2.2A} and \eqref{eqn:2.2B} imply that
\begin{equation}
\label{eqn:2.2C}
L(s) \subseteq L_1(s) \cup \cdots \cup L_m(s).
\end{equation}
Observe that 
\[
|L_i(s)| \le \alpha_i \cdot d_{m+1} \cdots d_n \cdot 
\dim\CC[z_1,\dots,\hat{z}_i,\dots,z_m]_{\le s}.
\]
Combining this with \eqref{eqn:HFCz} and \eqref{eqn:2.2C}, we obtain
$|L(s)|= O(s^{m-1})$.  
It follows that 
\begin{equation} \label{eqn:2.2D}
\dim\, (\CC[z]/I)_{\le s} = O(s^{m-1}).
\end{equation}
On the other hand, the inclusion $\CC[z_1,\dots,z_m] \subseteq
\CC[z]/I$ gives an inclusion 
\[
\CC[z_1,\dots,z_m]_{\le s} \subseteq
(\CC[z]/I)_{\le s},
\]
and then \eqref{eqn:HFCz} implies $\dim\, (\CC[z]/I)_{\le s} \ge
\frac{1}{m!}s^m + O(s^{m-1})$.  This contradicts \eqref{eqn:2.2D} and completes the proof.
\end{proof}

\section{Constructing an ordered basis} 
\label{sec:3}

In what follows we will use the following standard notation.

\begin{notation}\rm 
Given a set of polynomials $I\subseteq\CC[z_1,\dots,z_n]=\CC[z]$,
write
\[
\bV(I):=\{(a_1,\dots,a_n)\in\CC^n: p(a_1,\dots,a_n)=0 \hbox{ for all }
p\in S\},
\]
and given a set $V\subseteq\CC^n$, write
\[
\bI(V):=\{p\in\CC[z]: p(a_1,\dots,a_n)=0 \hbox{ for all }
(a_1,\dots,a_n) \in V\}.
\]
\end{notation}

Let $V\subseteq\CC^n$ be an affine algebraic variety of pure dimension
$m$ ($m\leq n$).  Here, ``pure'' means that all irreducible components
of $V$ have dimension $m$.  If we set
$I:=\bI(V)\subseteq\CC[z_1,\dots,z_n]$, then the coordinate ring
$\CC[V]$ of polynomial functions on $V$ satisfies
\[
\CC[V] \simeq \CC[z]/I \simeq \CC[z]_I.
\]
In what follows, we will use these isomorphisms to identify $\CC[V]$
with $\CC[z]_I$ and write $\CC[V]=\CC[z]_I$.

We will construct a special basis of $\CC[V]$ by doing interpolation
at infinity.  Identify $(a_1,\dots,a_n)\in\CC^n$ with
$[1:a_1:\cdots:a_n]\in\PP^n$; the hyperplane at infinity is
then 
\[
H_{\infty}:=\{[a_0:a_1:\cdots:a_n]\in\PP^n : a_0=0\}
\]
and we write $\CC^n\cup H_{\infty}=\PP^n$.  Denote by $\bar
V\subseteq\PP^n$ the projective closure of $V$, which may be computed
as follows.  If $I=\bI(V)\subseteq\CC[z] = \CC[z_1,\dots,z_n]$, let
\[
I^h: = \{p^h\in\CC[z_0,\dots,z_n]: p\in I\},
\] 
where $p(z) = \sum_{|\alpha|\leq d} c_{\alpha}z^{\alpha} \in \CC[z]$
of degree $d$ homogenizes to 
\[
p^h(z_0,z) := \sum_{|\alpha|\leq d} c_{\alpha} z_0^{d-|\alpha|}
z^{\alpha} \in \CC[z_0,z] = \CC[z_0,z_1,\dots,z_n].
\]
Then the projective closure $\bar V \subseteq \PP^n$ is given by
\[
\bar V = \bV(I^h) = \{[a_0:\cdots:a_n]\in\PP^n : p(a_0,\dots,a_n)=0
\hbox{ for all } p\in I^h\}.
\]
 
Note that $I^h$ is a homogeneous ideal (i.e., it is generated by
homogeneous polynomials).  For a homogeneous ideal
$J\subseteq\CC[z_0,\dots,z_n]$ we will write
\begin{eqnarray*}
J_t &=& \{p\in J: p \hbox{ is homogeneous, } \deg p=t\}, \quad
\hbox{and}\\ (\CC[z_0,\dots,z_n]/J)_t &=& \CC[z_0,\dots,z_n]_t/J_t.
\end{eqnarray*}
 
We will assume that $V$ has the following properties: 
\begin{equation}
\label{eqn:properties}
\begin{aligned}
(0)\ &\text{$V$ is pure of dimension $m$ and has degree $d$.}\\
(1)\ &R:=\CC[z_1,\dots,z_m]\subseteq\CC[V]\ \text{is a Noether
    normalization as above.}\\ 
(2)\ &\bar{V}\cap P\ \text{consists of $d$ distinct points, where
    $\bar V$ is the projective}\\[-3pt] 
&\text{closure of $V$ in $\PP^n$ and
    $P=\bV(z_0,\dots,z_{m-1})\subseteq\PP^n$.}\\ 
(3)\ & \text{If $\bar V\cap P = \{p_1,\dots,p_d\}$, with
    $p_i=[0:\cdots:0:p_{im}:\cdots:p_{in}]$,}\\[-3pt] 
&\text{then for each $i$, $p_{im}\neq 0$.}
\end{aligned}
\end{equation}

Note that $\bar V\subseteq\PP^n$ is pure of dimension $m$ and has
degree $d$, while $P\subseteq\PP^n$ is a linear space of dimension
$n-m$ and has degree $1$.  Since $\bar V\cap P$ is finite by property
(3), Bezout's theorem implies that $\bar V\cap P$ consists of $d \cdot
1 = d$ points counted with multiplicity.  Property (3) then implies
that the multiplicities of the $p_i$ are all one, so that
\[
\bV(I^h+\langle z_0,\dots,z_{m-1}\rangle) = \{p_1,\dots,p_d\}
\] 
as subschemes of $\PP^n$.  

It follows that the homogeneous ideals $I^h+\langle
z_0,\dots,z_{m-1}\rangle$ and $\bI(\{p_1,\dots,p_d\})$ define the same
subscheme of $\PP^n$.  Hence there is an integer $t_0 \ge 0$ such that
\begin{eqnarray*}
(I^h+\langle z_0,\dots,z_{m-1}\rangle)_t &=&
  (\bI(\{p_1,\dots,p_d\}))_t \\ &=& \{f\in\CC[z_0,\dots,z_n]_t :
  f(p_i)=0, \ \forall\; i=1,\dots,d\}
\end{eqnarray*}
when $t\geq t_0$ (see \cite{hartshorne:algebraic}, II.5). 

A polynomial $f\in\CC[z_0,\dots,z_n]_t$ gives a function on
\begin{equation}\label{eqn:3.2new}
U_m=\{z=[z_0:z_1:\cdots:z_n]\in\PP^n: z_m\neq 0\},
\end{equation}
via 
$[a_0:\cdots:a_n]\mapsto a_m^{-t}f(a_0,\dots,a_n)$. It is easy to see that the computation is independent of homogeneous coordinates.  For convenience this local evaluation will be denoted by $f(a)$.

\begin{lemma}
The map $\CC[z_0,\dots,z_n]_t\to \CC^d$ given by
$f\mapsto(f(p_1),\dots,f(p_d))$ is onto for $t\gg0$. 
\end{lemma}

\begin{proof}
By property (3) of \eqref{eqn:properties}, the points $p_1,\dots,p_d$ are in the affine
chart $U_m$ given by \eqref{eqn:3.2new}.  For each $i=1,\dots,d$ and
$p_i=[0:\cdots:0:1:u_{i(m+1)}:\cdots:u_{in}]$, put $q_i :=
(0,\dots,0,u_{i(m+1)},\dots,u_{in})\in\CC_m^{n}$, where $\CC_m^n$
denotes affine space with coordinates
$(z_0,\dots,z_{m-1},z_{m+1},\dots,z_n)$.  It is standard that one can
find interpolating polynomials $w_1,\dots,w_d$ in
$\CC[z_0,\dots,z_{m-1},z_{m+1},\dots,z_n]$ such that
$w_i(q_j)=\delta_{ij}$.

Pick any $t\geq\max(\deg w_1,\dots,\deg w_d)$ and set
\begin{equation}
\label{eqn:2.5c}
v_i:= z_m^tw_i(z_0/z_m,\dots,z_{m-1}/z_m,z_{m+1}/z_m,\dots,z_n/z_m).
\end{equation}
This is a homogeneous polynomial of degree $t$ in $z_0,\dots,z_n$ and its evaluation on $U_m$
satisfies $v_i(p_j)=\delta_{ij}$.  For each $i$, the polynomial
$v_i\in\CC[z_0,\dots,z_n]$ evaluates to the standard basis vector
$(0,\dots,0,1,0,\dots,0)=e_i\in\CC^d$ (the $1$ is in the $i$-th slot),
so the map is onto.
\end{proof}

\begin{corollary}\label{cor:3.3}
For $t\gg0$, we have an exact sequence
\[
0\longrightarrow (I^h+\langle z_0,\dots,z_{m-1}\rangle)_t
\longrightarrow\CC[z_0,\dots,z_n]_t\longrightarrow\CC^d\longrightarrow
0. 
\]
Thus there are polynomials $v_1,\dots,v_d\in\CC[z_0,\dots,z_n]_t$,
unique up to elements of $(I^h+\langle z_0,\dots,z_{m-1}\rangle)_t$,
such that $v_i(p_j)=\delta_{ij}$. \qed
\end{corollary}

Now fix such a $t$ and let $\calS:=\CC[z_0,\dots,z_n]/( I^h+\langle
z_0,\dots,z_{m-1}\rangle)$.  If we regard the polynomials
$v_1,\dots,v_d$ in the above corollary as elements of $\calS_t$, then
they have the following properties:
\begin{equation} 
\label{eqn:2.6c}
v_i^2 = z_m^tv_i \hbox{ for all }i=1,\dots,d; \quad \hbox{and }
v_iv_j =0 \hbox{ whenever }i\neq j.
\end{equation}

\begin{lemma}
\label{lem:Stau}
For any $\tau\geq t$, the polynomials
$\{z_m^{\tau-t}v_i\}_{i=1}^d$ form a basis of $\calS_{\tau}$. 
\end{lemma}

\begin{proof}
The construction \eqref{eqn:2.5c} applied to $\tau$ (in place of $t$) gives the
additional powers of $z_m$. 
\end{proof}

When we consider the $v_i$'s as polynomials in 
$\CC[z_0,\dots,z_n]/(I^h +\langle z_0\rangle)$, we have
\begin{eqnarray}
v_i^2 &=& z_m^tv_i + \sum_{k=1}^{m-1}
z_kH_k(z_1,\dots,z_n), \label{eqn:1.2} \\ 
v_iv_j &=& \sum_{k=1}^{m-1} z_kQ_k(z_1,\dots,z_n), \label{eqn:1.3}
\end{eqnarray}
where for each $k$, $H_k(z_1,\dots,z_n)$ and $Q_k(z_1,\dots,z_n)$ are
homogeneous polynomials of degree $2t-1$. 

The next step is to translate the $v_i$ into polynomials $\bv_i$ in
$\CC[V]$, paying careful attention to their degrees and the analogs
of \eqref{eqn:1.2} and \eqref{eqn:1.3}.  Let $\CC[V]_{\leq t} =
\CC[z]_{I\le t}$ be the collection of normal forms of degree $\leq
t$, and let $\CC[V]_{=t}$ be those that are homogeneous of degree $t$.

\begin{lemma} \label{lem:3.5}
We have $\CC[V]_{=t}\simeq\CC[V]_{\leq t}/\CC[V]_{\leq
  t-1}\simeq(\CC[z_0,\dots,z_n]/(I^h+\langle z_0\rangle))_t$. 
\end{lemma}

\begin{proof}
Writing a normal form as a sum of homogeneous components gives the
direct sum decomposition $\CC[V]_{\le t} = \CC[V]_{=t} \oplus
\CC[V]_{\leq t-1}$, and the first isomorphism follows immediately.

For the second, the map $p \mapsto z_0^t\hskip1pt
p(z_1/z_0,\dots,z_n/z_0)$ induces an isomorphism
\[
\CC[V]_{\le t} \simeq (\CC[z]/I)_{\le t} \simeq (\CC[z_0,z]/I^h)_t
\]
(see \cite{coxlittleoshea:ideals}, \S9.3).  This isomorphism sends
$\CC[V]_{\leq t-1} \subseteq \CC[V]_{\leq t}$ to
$z_0(\CC[z_0,z]/I^h)_{t-1}$, so that we get an isomorphism
\[
\CC[V]_{\le t}/\CC[V]_{\leq t-1} \simeq
(\CC[z_0,z]/I^h)_t/z_0(\CC[z_0,z]/I^h)_{t-1} \simeq
(\CC[z_0,z]/(I^h+\langle z_0\rangle))_t.\qedhere
\]
\end{proof}

\begin{remark} \label{rem:3.6} \rm
Note that multiplication in $\CC[z_0,\dots,z_n]/(I^h+\langle
z_0\rangle)$ corresponds to linear maps $\hsta:\CC[V]_{=t} \times
\CC[V]_{=s}\to\CC[V]_{=s+t}$, where to get $p{\hsta} q$, we compute
$p{*}q$ (the normal form of $pq$) and then take the homogeneous
part of degree $s+t$.
\end{remark}

\begin{lemma}
\label{lem:2.4}
For each $i=1,\dots,d$, there is a polynomial $\bv_i\in\CC[V]_{=t}$
that satisfies the following equations in $\CC[V]$:
\begin{enumerate}
\item[{\rm(1)}] $\bv_i{*}\bv_i = z_m^t{*}\bv_i + \sum_{k=1}^{m-1} z_k{*}h_{k} +
  h_0$ with $\deg(h_{k})\leq 2t-1$ for each $k=0,\dots,m-1$.
\item[{\rm(2)}] $\bv_i{*}\bv_j = \sum_{k=1}^{m-1} z_k{*}q_k + q_0$ if $i\neq j$
  with $\deg(q_k)<2t-1$ for each $k$.
\end{enumerate}
\end{lemma}

\begin{remark}\rm
Since $\CC[V]$ is identified with the space $\CC[z]_I$ of normal
forms, the products involving $*$ in
Lemma~\ref{lem:2.4} represent multiplication of
polynomials followed by reduction to normal form.
\end{remark}

\begin{proof}
Given $v_i\in (\CC[z_0,\dots,z_n]/(I^h+\langle
z_0\rangle))_t$, let $\bv_i$ be the element of $\CC[V]_{=t}$
given by the isomorphism in Lemma \ref{lem:3.5}.  For each
$k=1,\dots,m-1$, let $h_k\in\CC[V]_{=2t-1}$ be the element
corresponding to $H_k\in(\CC[z_0,\dots,z_n]/(I^h+\langle
z_0\rangle))_{2t-1}$ in \eqref{eqn:1.2}.  Then by
\eqref{eqn:1.2}, the polynomial
\[
\bv_i{\hsta} \bv_i - z_m^t{\hsta}\bv_i - \sum_{k=1}^{m-1} z_k{\hsta} h_k
\in\CC[V]_{=2t} 
\] 
corresponds to the zero polynomial in
$\left(\CC[z_0,\dots,z_n]/(I^h+\langle z_0\rangle)\right)_{2t}$, so it
must be zero in $\CC[V]_{=2t}$.  (Here, $\hsta$ is as in Remark \ref{rem:3.6}.)  Thus the polynomial $h_0:=\bv_i{*}\bv_i - z_m^t{*}\bv_i -
\sum_{k=1}^{m-1} z_k{*}h_k$ is in $\CC[V]_{\leq 2t-1}$.  This proves (1).

A similar argument applied to \eqref{eqn:1.3} proves (2).
\end{proof}

In what follows, we use the notation
\[
z^{\alpha} = z_1^{a_1}\cdots z_{m-1}^{a_{m-1}}, \quad z^{\beta} =
z_{m+1}^{b_1}\cdots z_n^{b_n}.
\]
Define the finite set of monomials
\begin{equation}
\label{eqn:calB}
\calB:=\{z_m^lz^\beta\notin\langle\LT(I)\rangle,\ l+|\beta|\leq
t-1\}\subseteq\CC[V].
\end{equation}

\begin{proposition}
\label{prop:span}
$\CC[V]$ is spanned over $\CC$ by the homogeneous polynomials 
\begin{eqnarray*}
&(*) & z^{\alpha}z_m^l{*}z^{\beta}: \quad \alpha\in\ZZ_{\geq 0}^{m-1},
  \ z_m^lz^{\beta}\in\calB,  \ \hbox{and} \\ 
&(**)& z^{\alpha}z_m^l{*}\bv_i: \quad \alpha\in\ZZ_{\geq 0}^{m-1},
  \ l\geq 0, \ i=1,\dots,d.  
\end{eqnarray*}
\end{proposition}

\begin{remark} \label{rmk:nfbasis}\rm
Note that $z^{\alpha}z_m^l$ is a normal form by
Proposition~\ref{prop:2.2a}, while the products
$z^{\alpha}z_m^lz^{\beta}$ and $z^{\alpha}z_m^l\bv_i$ may fail
to be normal forms.  This explains why the proposition uses
$z^{\alpha}z_m^l{*}z^{\beta}$ and $z^{\alpha}z_m^l{*}\bv_i$. 
\end{remark}

\begin{proof}
To simplify the proof, we will omit the $*$ when multiplying normal
forms.  It suffices to show that any monomial
$z^{\alpha}z_m^lz^{\beta}\notin\langle\LT(I)\rangle$ can be expressed
as a linear combination of elements of $(*)$ and $(**)$. 

We will prove this by induction on $s = |\alpha|+l+|\beta|$.  Suppose
$z^{\alpha}z_m^lz^{\beta} \notin \langle\LT(I)\rangle$ with $s\leq
t-1$.  Then $|\alpha|+l+|\beta|\leq t-1$, so that
$z_m^lz^{\beta}\in\calB$.  Hence the monomial is in ($*$), which
proves the base case.

Next, assume $s\geq t$ and that $\CC[V]_{\leq s-1}$ is spanned by the
polynomials ($*$) and ($**$) of degree $\leq s-1$.  Take
$z^{\alpha}z_m^lz^{\beta} \notin \langle\LT(I)\rangle$ of degree $s$.  No factor of this monomial is in the ideal either; in particular, 
$z_m^lz^{\beta} \notin \langle\LT(I)\rangle$.  If $l+|\beta|\leq t-1$,
then $z_m^lz^{\beta}\in\calB$ and therefore $z^{\alpha}z_m^lz^{\beta}$
is an element of the form $(*)$.

Otherwise, $\tau:=l+|\beta|\geq t$.  By Lemma~\ref{lem:Stau}, we have
an equation
\[
z_m^lz^{\beta} = \sum_{i=1}^d a_i z_m^{\tau-t} v_{i} +
\sum_{j=0}^{m-1}z_jH_j(z_0,z) + H(z_0,z),
\]
in $\CC[z_0,z]$, where $a_i \in \CC$, $\deg H_j = \tau-1$ and $H \in
I^h$.  If we dehomogenize by setting $z_0 = 1$, we obtain
\[
z_m^lz^{\beta} = \sum_{i=1}^d a_i z_m^{\tau-t} v_{i} +
\sum_{j=1}^{m-1}z_j h_j(z) + h_0(z)
\]
in $\CC[z]/I$, where $a_i \in \CC$ and $\deg h_j \le \tau-1$.  We can
multiply by $z^\alpha$ to obtain
\[
z^\alpha z_m^lz^{\beta} = \sum_{i=1}^d a_i z^\alpha z_m^{\tau-t} v_{i} +
\sum_{j=1}^{m-1}z_j (z^\alpha h_j(z)) + z^\alpha h_0(z)
\]
in $\CC[z]/I$.  Using the isomorphism $\CC[V] = \CC[z]_I \simeq
\CC[z]/I$, this becomes
\[
z^\alpha z_m^lz^{\beta} = \sum_{i=1}^d a_i z^\alpha z_m^{\tau-t}
\bv_{i} + \sum_{j=1}^{m-1}z_j(z^\alpha h_j(z)) + z^\alpha
   h_0(z).
\]
in $\CC[V]$.  The first sum is a linear combination of elements of
the form $(**)$.  For the second sum, note that $\deg(z^{\alpha}h_j)\leq s-1$
for each $j=1,\dots,m-1$.  By the inductive hypothesis, this means
that $z^{\alpha}h_j$ is a linear combination of terms in $(*)$ and
$(**)$, and therefore $z_jz^{\alpha}h_j$ is too, by definition.
Finally, $\deg(z^{\alpha}h_0)\leq s-1$, and again by induction,
$z^{\alpha}h_0$ is a linear combination of terms in $(*)$ and
$(**)$.
\end{proof}

The following is an immediate corollary of the above proof.

\begin{corollary}
\label{cor:2.6}
$\CC[V]_{\leq s}$ is spanned over $\CC$ by the polynomials in $(*)$
and $(**)$ of degree $\leq s$. \qed
\end{corollary}

Now that we have a spanning set, the next step in constructing the
desired basis for $\CC[V]$ is to show that the elements of the form $(**)$ are
linearly independent over $\CC$.  These elements are monomials in
$z_1,\dots,z_m$ multiplied by one of $\bv_1,\dots,\bv_d$.  Since the
inclusion $\CC[z_1,\dots,z_m] \subseteq\CC[V]$ makes $\CC[V]$ into a
module over $R=\CC[z_1,\dots,z_m]$, we can verify linear independence
by showing the following.

\begin{theorem}
\label{thm:freemodule}
The polynomials $\bv_1,\dots,\bv_d$ generate a free $R$-submodule of
$\CC[V]$. 
\end{theorem}

\begin{proof} 
We first observe that since $V$ has dimension $m$ and degree $d$, we
have
\begin{equation}
\label{eqn:hilbert} 
\dim\CC[V]_{\leq s}=\frac{d}{m!}s^m + O(s^{m-1})
\end{equation}
(see e.g.\ \cite{coxlittleoshea:ideals}, \S9.3).
Now let $M:=\sum_{i=1}^d R\bv_i$ and $N: = \sum_{\calB}
\CC[z_1,\dots,z_{m-1}] z_m^lz^{\beta}$, and for $s\geq t$ define
\[
M_{\leq s} := \sum_{i=1}^d R_{\leq s-t}\bv_i,\quad N_{\leq s} :=
\sum_{\calB}\CC[z_1,\dots,z_{m-1}]_{\leq
  s-l-|\beta|}z_m^lz^{\beta}. 
\]
The corollary implies $\CC[V]_{\leq s} = M_{\leq s} + N_{\leq s}$.
Using \eqref{eqn:HFCz}, one easily obtains
\[
\dim N_{\leq s} \leq \frac{|\calB|}{(m-1)!}s^{m-1} + O(s^{m-2}). 
\]
Combining this with \eqref{eqn:hilbert} and $\CC[V]_{\leq s} = M_{\leq
  s} + N_{\leq s}$ yields
\[
\dim M_{\leq s} = \frac{d}{m!}s^m+ O(s^{m-1}).
\]

Suppose there is a nontrivial relation 
\begin{equation} \label{eqn:fv}
f_1\bv_1+\cdots+f_d\bv_d=0, \quad  f_i\in R,\ \hbox{not all } f_i=0. 
\end{equation} 
Let $D=\max\{\deg f_1,\dots,\deg f_d\}$ and take a large integer
$s\geq t + D$.  There is an exact sequence 
\[
0 \lto K_s\lto R_{\leq s-t}^d \stackrel{\varphi_s}{\lto} M_{\leq s}\lto 0
\]
where $\varphi_s:R_{\leq s-t}^d\to M_{\leq s-t}$ is given by
$\varphi(g_1,\dots,g_d) = \sum_i g_i\bv_i$ and $K_s:=\ker\varphi_s$.  We
have $R_{s-t-D}\cdot(f_1,\dots,f_d) \subseteq R_{\leq s-t}^d$,  
so by \eqref{eqn:fv}, 
\[
R_{s-t-D}(f_1,\dots,f_d)\subseteq K_s.
\]
Since $(f_1,\dots,f_d)\neq(0,\dots,0)$ we have $K_s\neq 0$ and so $\dim
K_s\geq\dim R_{s-t-D}$.  Thus
\[
\dim R_{\leq s-t}^d \ = \ \dim M_{\leq s} + \dim K_s \ \geq \  \dim
M_{\leq s} + \dim R_{s-t-D}. 
\]
A Hilbert function calculation then gives the inequality
\[
\frac{d}{m!}s^m+ O(s^{m-1}) \geq \Bigl(\frac{d}{m!}s^m +
O(s^{m-1})\Bigr) + \Bigl(\frac{1}{m!}(s-t-D)^m + O(s^{m-1}) \Bigr), 
\]
so that $\frac{1}{m!}s^m\leq O(s^{m-1})$, a contradiction.  This says
that no equation of the form \eqref{eqn:fv} can hold, and so
$\bv_1,\dots,\bv_d$ are free over $R$.
\end{proof}

We now construct the sought-after ordered basis for $\CC[V]$.  

\begin{definition}\rm
\label{def:basisCprec}  
The polynomials given by $(*)$ and $(**)$ span $\CC[V]$ by
Proposition~\ref{prop:span}, and those from $(**)$ are linearly
independent by Theorem~\ref{thm:freemodule}.  We first create a basis of $\CC[V]$ by adjoining a sufficient number of elements of the form $(*)$ to those of the form $(**)$.  List those of the form $(*)$ in grevlex order and discard any monomial that is linearly dependent with respect to elements of the form $(**)$ together with previous elements of $(*)$; otherwise keep it.  This yields the basis $\calC$ of $\CC[V]$.  We define an ordering $\prec$ on $\calC$ as
follows.  First, order the elements by total degree; then for a fixed degree $s$,
\begin{itemize}
\item let elements of $(*)$ precede elements of $(**)$;
\item let $z^{\alpha}z_m^l{*}\bv_i\prec z^{\hat\alpha}z_m^{\hat
  l}{*}\bv_j$ if $z^{\alpha}z_m^l$ precedes $z^{\hat\alpha}z_m^{\hat l}$
  according to grevlex; 
\item let $z^{\alpha}z_m^l{*}\bv_i\prec z^{\alpha}z_m^l{*}\bv_j$ if $i<j$;
  and
\item let elements of the form $(*)$ be ordered according to grevlex.
\end{itemize}
\end{definition}  

It is easy to see that the elements of $\calC$ of degree $\le s$ form
a basis of $\CC[V]_{\le s}$.  The Chebyshev constants defined in
Section~\ref{sec:cheby} will use the ordered basis of $\CC[V]$ given
in Definition~\ref{def:basisCprec}.

\medskip

We conclude this section by computing some examples of  $\calC$ and
$\prec$. 

\begin{example}\rm
Let $V=\{z=(z_1,\dots,z_n)\in\CC^n : z_{m+1}=z_{m+2}=\cdots
=z_n=0\}$.  The Noether normalization is the identity,
$\CC[z_1,\dots,z_m] = \CC[V]$, and in the notation of
\eqref{eqn:properties}, $\bar V\cap
P=\{[0:\cdots:0:1:0:\cdots:0]\}$, where the $1$ is in the $m$-th slot.
We take $v_1=\bv_1=1$ (so $t=0$).  The basis $\calC$ consists of the
monomials in $\CC[z_1,\dots,z_m]$, which are elements of the form
$(**)$, ordered by grevlex.  There are no elements of the form $(*)$
in this case. 
\end{example}

\begin{example}
\label{ex:sphere}\rm
Let $V$ be the complexified sphere in $\CC^3$, i.e., the algebraic
surface given by the equation $z_1^2+z_2^2+z_3^2=1$.  A basis of
$\CC[V]$ is given by all monomials not in $\langle z_3^2\rangle$,
i.e., 
\[
1,z_1,z_2,z_3,z_1^2,z_1z_2,z_1z_3,z_2^2,z_2z_3,z_1^3,\dotsc.
\]
The Noether normalization is $\CC[z_1,z_2]\subseteq\CC[V]$.

In $\PP^3$, $\bar V$ is given by all points $[z_0:z_1:z_2:z_3]$
satisfying $z_1^2+z_2^2+z_3^2=z_0^2$, and $P=\{z_0=z_1=0\}$.  The
points of $\bar V\cap P$  are then $p_1=[0:0:1:-i]$ and
$p_2=[0:0:1:i]$.   Thus \eqref{eqn:properties} is satisfied.

Interpolating polynomials are $v_1=\frac{1}{2}(z_2+iz_3)$ and
$v_2=\frac{1}{2}(z_2-iz_3)$.  In this case $t=1$ so that $\bv_1=v_1$
and $\bv_2=v_2$.  The first few elements of the basis $\calC$, ordered
by $\prec$, are 
\[
1,z_1,\bv_1,\bv_2,z_1^2,z_1\bv_1,z_1\bv_2,z_2\bv_1,z_2\bv_2,z_1^3,\dotsc.
\]
Basis elements of the form $(*)$ are $z_1^k$ while those of the form
$(**)$ are $z_1^{\alpha_1}z_2^{\alpha_2}\bv_i$.  (Note that since
$l<t=1$, no factors of the form $z_2^l$ appear in $(*)$). 
\end{example}

\begin{example}
\label{ex:hypersurface}\rm
When $V=\bV(f)\subseteq\CC^n$ is a hypersurface given by
$f\in\CC[z_1,\dots,z_n]$, we can generalize Example~\ref{ex:sphere} by
computing the basis $\calC$ rather explicitly.  We assume that $f$ is
a product of distinct irreducible polynomials, so that $I = \bI(V) =
\langle f\rangle$.  We also assume that $\LT(f) = z_n^d$ where
$d=\deg(f)$.  This ensures that $\CC[z_1,\dots,z_{n-1}] \subseteq
\CC[V]$ is a Noether normalization.

Let $F:=f^h\in\CC[z_0,\dots,z_n]$ be the homogenization of $f$; then
in $\PP^n$, $\bar V=\bV(F)$ and $I^h = \langle F\rangle$.  If the
properties \eqref{eqn:properties} hold, then
$\bV(F,z_0,\dots,z_{n-2})\subseteq\PP^n$ consists of $d$ distinct
points, all with $z_{n-1}\neq 0$, given by $[0:\cdots:1:\beta_i]$ for
$i = 1,\dots,d$.

Separating the terms of $F$ containing only the variables
$z_{n-1},z_n$ from the others, we have 
\begin{equation}
\label{eqn:FG}
F(z) \  = \  G(z_{n-1},z_n)  \ +
\  \sum_{l=0}^{n-2}z_lH_l(z_0,\dots,z_n),
\end{equation}
where $\deg(G)=d$ and $\deg H_l=d-1$ for each $l=0,\dots,n-2$.  Thus
$G(1,\beta_i) = 0$ for $i = 1,\dots,d$.

In the notation of earlier in the section, we have
\[
\begin{aligned}
\calS &= \CC[z_0,\dots,z_n]/(I^h+\langle z_0,\dots,z_{n-2}\rangle) =
\CC[z_0,\dots,z_n]/\langle F(z),z_0,\dots,z_{n-2}\rangle\\
&= \CC[z_0,\dots,z_n]/\langle G(z_{n-1},z_n),z_0,\dots,z_{n-2}\rangle\\
&\simeq \CC[z_{n-1},z_n]/\langle G(z_{n-1},z_n)\rangle,
\end{aligned}
\]
where the second line uses \eqref{eqn:FG} and the third uses 
 the map
\[
p(z_0,z_1,\dots,z_n) \mapsto p(0,\dots,0,z_{n-1},z_n). 
\]
We factor $G(z_{n-1},z_n) = \prod_{i=1}^d(z_n-\beta_iz_{n-1}) =
\prod_{i=1}^d l_i(z_{n-1},z_n)$.  Note that $\beta_i\neq\beta_j$ if
$i\neq j$.  For each $i=1,\dots,d$, define
\begin{equation}
\label{eqn:2} 
v_i(z_{n-1},z_n) = \prod_{j\neq i}
\frac{l_j(z_{n-1},z_n)}{l_j(1,\beta_i)}.
\end{equation}  
Then $\deg(v_i)=d-1$ for each $i$, and clearly
\begin{equation}
\label{eqn:3} 
v_i(1,\beta_j)=\begin{cases} 0 & \hbox{if }
j\neq i\\ 1&\hbox{if } j=i.\end{cases} 
\end{equation}

Note that when $f = z_1^2+z_2^2+z_3^2-1$ as in
Example~\ref{ex:sphere}, we have the points $[0:0:1:-i]$ and
$[0:0:1:i]$.  Then $G = z_2^2+z_3^2 = (z_3+iz_2)(z_3-iz_2) =
l_1l_2$ and the formula for $v_1$ reduces to
\[
v_1 = \frac{l_2(z_2,z_3)}{l_2(1,-i)} = \frac{z_3-iz_2}{-2i} =
{\textstyle\frac12}(z_2+iz_3),
\]
in agreement with Example~\ref{ex:sphere}.  The formula for $v_2$
works similarly.

By \eqref{eqn:3}, $v_1,\dots,v_d$ satisfy Lemma~\ref{lem:Stau} with $t
= d-1$.  Since the $v_i$ only involve $z_{n-1},z_n$ and are normal
forms with respect to grevlex (having degree $\le d-1$ in $z_n$), we can take $\bv_i = v_i$ in
Lemma \ref{lem:2.4}.  Thus $\bv_1,\dots,\bv_d$ are defined by
\eqref{eqn:2} and have degree $d-1$.

The next step is to identify the set $\calB$ from \eqref{eqn:calB}.  
Since $m = n-1$, the monomials $z^\alpha$ and $z^\beta$ from
Proposition \ref{prop:span} are 
\[
z^\alpha = z_1^{a_1} \cdots z_{n-2}^{a_{n-2}},\quad z^\beta = z_n^b.
\]
In this notation, a monomial in $z_1,\dots,z_n$ is written $z^\alpha
z_{n-1}^l z_n^b$.  Since the $\bv_i$ have degree $t = d-1$ and $\langle \lt(I) \rangle =
\langle \lt(f) \rangle = \langle z_n^d \rangle$, it follows that
\eqref{eqn:calB} becomes
\[
\calB = \{z_{n-1}^l z_n^b \notin \langle z_n^d \rangle : l + b \le
d-2\} = \{z_{n-1}^l z_n^b : l + b \le d-2\}.
\]
Hence the collections $(*)$ and $(**)$ from Proposition
\ref{prop:span} are 
\begin{equation}
\label{eqn:hypersurface}
\begin{aligned}
(*) &\  z^{\alpha}z_{n-1}^lz_n^{b}: \quad \alpha\in\ZZ_{\geq 0}^{n-2},
  \ l + b \le d-2,  \ \hbox{and} \\ 
(**) &\  z^{\alpha}z_{n-1}^l\bv_i: \quad \alpha\in\ZZ_{\geq 0}^{n-2},
  \ l\geq 0, \ i=1,\dots,d.  
\end{aligned}
\end{equation}
These products are all normal forms, so no $*$ is needed in the multiplications.

The nicest feature of the hypersurface case is that the
basis $\calC$ consists \emph{precisely} of the polynomials in
\eqref{eqn:hypersurface}.  They span by Proposition \ref{prop:span},
so we only need to prove linear independence.  The polynomials in
$(**)$ are linearly independent by Theorem~\ref{thm:freemodule}, and
those in $(*)$ are linearly independent since they are normal-form
monomials.  Hence it remains to study an equation of the form
\[
\text{linear combination of } z^{\alpha}z_{n-1}^lz_n^{b} =
\text{linear combination of } z^{\alpha}z_{n-1}^l\bv_i.
\]
The left-hand side has degree $\le d-2$ in $z_{n-1},z_n$ and the
right-hand side has degree $\ge d-1$.  This forces the linear
combinations to be trivial, and linear independence follows.

To summarize: when $V=\bV(f)$ is a hypersurface of degree $d$, the
$\bv_i$'s are polynomials of degree $d-1$ that we can compute
explicitly in terms of $f$, and the elements of $(*)$ consist of all
monomials $z_1^{\alpha_1}\cdots z_{n-1}^{\alpha_{n-1}}z_n^{\alpha_n}$
with $\alpha_{n-1}+\alpha_{n}\leq d-2$.
\end{example}
 
\section{Weakly submultiplicative functions} 
\label{sec:4}

In \cite{bloomlev:weighted}, Bloom and Levenberg observed that the
main properties of Zaharjuta's directional Chebyshev constants
followed from the submultiplicative property of sup norms of Chebyshev
polynomials, and could be recast rather abstractly as properties of
submultiplicative functions on integer tuples.  We verify here that
these properties still hold under slightly weaker conditions.  The
arguments are those of Zaharjuta's paper \cite{zaharjuta:transfinite}
with minor adjustments.  We will apply these results concretely in the
next section.

\begin{definition} 
\label{def:Y} \rm
Let $m$ be a positive integer.  A non-negative function $Y:\ZZ_{\geq
  0}^m\to\RR_{\geq 0}$ is said to be \emph{weakly submultiplicative}
if there is a finite subset $\calF$ of $\ZZ_{\geq 0}^m$ such that:
\[
\text{\sl For all $\alpha,\beta\in\ZZ_{\geq 0}^m$ there exists
  $\gamma\in\calF$ such that $Y(\alpha+\beta+\gamma)\leq
  Y(\alpha)Y(\beta)$.} 
\]
$Y$ has \emph{subexponential growth} if for some $C,r>0$ we have
$Y(\alpha)\leq Cr^{|\alpha|}$ for all $\alpha$. 
\end{definition}

\begin{remark}[cf.\ \cite{bloomlev:weighted}]\rm
When $Y(\alpha+\beta)\leq Y(\alpha)Y(\beta)$, i.e.,
$\calF=\{(0,\dots,0)\}$, $Y$ is called \emph{submultiplicative}.  A
submultiplicative function automatically has subexponential growth: if
$\alpha=(\alpha_1,\dots,\alpha_m)$ then
\[
Y(\alpha) = Y\Big(\sum_{k=1}^m\alpha_ke_k\Big) \leq \prod_{k=1}^m
Y(e_k)^{\alpha_k} \leq r^{|\alpha|}, 
\]
where $e_k$ is the $k$-th coordinate vector and $r=\max_k Y(e_k)$.
It seems that weak submultiplicativity should also imply
subexponential growth, but the above argument runs into some technical
difficulties.
\end{remark}

Let 
\[
\Sigma_m:=\bigl\{\theta=(\theta_1,\dots,\theta_m)\in\RR^m :
\theta_i\geq 0 \ \forall\,i, \ {\textstyle\sum_i} \theta_i=1\bigr\}
\] 
denote the simplex in $\RR^m$, and let $\Sigma_m^{\circ} :=
\bigl\{\theta\in\Sigma_m : \theta_i>0\ \forall\, i\}$ be its interior.

\begin{lemma} 
\label{lem:3.5a}
Let $Y:\ZZ_{\geq 0}^m\to\RR_{\geq 0}$ be weakly submultiplicative with
subexponential growth.  For all $\theta\in\Sigma_m^{\circ}$, the limit
$\displaystyle T(\theta) :=
\lim_{\substack{|\alpha|\to\infty\\ \frac{\alpha}{|\alpha|}\to\theta}}
Y(\alpha)^{\frac{1}{|\alpha|}}$ exists.
\end{lemma}

\begin{proof}
Let $\{\alpha_{(j)}\}$ and $\{\tilde \alpha_{(j)}\}$ be sequences in
$\ZZ_{\geq 0}^m$ such that
$\frac{\alpha_{(j)}}{|\alpha_{(j)}|},
\frac{\tilde\alpha_{(j)}}{|\tilde\alpha_{(j)}|}\to\theta$ as $j\to\infty$ 
and  
\begin{eqnarray*}
\lim_{j\to\infty} Y(\alpha_{(j)})^{\frac{1}{|\alpha_{(j)}|}}&=&
\liminf_{{|\alpha|\to\infty, \frac{\alpha}{|\alpha|}\to\theta}}
Y(\alpha)^{\frac{1}{|\alpha|}}:=L_1, \\ 
\lim_{j\to\infty}
Y(\tilde\alpha_{(j)})^{\frac{1}{|\tilde\alpha_{(j)}|}}&=&
\limsup_{{|\alpha|\to\infty, \frac{\alpha}{|\alpha|}\to\theta}}
Y(\alpha)^{\frac{1}{|\alpha|}}:=L_2. 
\end{eqnarray*}
To prove the lemma it is sufficient to show that $L_2\leq L_1$.  By
passing to subsequences we may assume that
$\frac{|\tilde\alpha_{(j)}|}{|\alpha_{(j)}|}\to\infty$ as
$j\to\infty$.   

Let $q_j$ denote the largest non-negative integer for which all the
components of $r_{(j)}:=\tilde\alpha_{(j)}-q_j\alpha_{(j)}$ are
non-negative.  We claim that
\begin{equation} 
\label{eqn:3.2b}
 \frac{q_j|\alpha_{(j)}|}{|\tilde\alpha_{(j)}|}\to
 1,\ \frac{|r_{(j)}|}{|\tilde\alpha_{(j)}|}\to 0 \quad \hbox{as }
 j\to\infty. 
\end{equation}
Write $\alpha_{(j)}=|\alpha_{(j)}|(\theta +\epsilon_{(j)})$ and
$\tilde\alpha_{(j)}=|\tilde\alpha_{(j)}|(\theta+\tilde\epsilon_{(j)})$
where $\epsilon_{(j)},\tilde\epsilon_{(j)}\to 0$ as $j\to\infty$.  
A calculation in components shows that
\begin{equation}
\label{eqn:3.3b}
\tilde\alpha_{(j)\nu} =
\frac{|\tilde\alpha_{(j)}|}{|\alpha_{(j)}|}\left( 1  +
\frac{|\alpha_{(j)}|}{\alpha_{(j)\nu}}(\tilde\epsilon_{(j)\nu}-
\epsilon_{(j)\nu})
\right) \alpha_{(j)\nu}  \quad \hbox{for each } \nu=1,\dots,m, 
\end{equation}
where we write $\alpha_{(j)}=(\alpha_{(j)1},\dots,\alpha_{(j)m})$,
etc.  For any $\nu$, we have
\[
\frac{|\alpha_{(j)}|}{\alpha_{(j)\nu}}(\tilde\epsilon_{(j)\nu}-
\epsilon_{(j)\nu}) \longrightarrow \frac{1}{\theta_{\nu}}(0-0) \ = \ 0
\quad \hbox{as } j\to\infty.
\]
(Here we use the fact that $\theta\in\Sigma_m^{\circ}$, so
$\theta_{\nu}\neq 0$.) This says that given $\epsilon>0$, the quantity
in parentheses on the right-hand side of \eqref{eqn:3.3b} exceeds
$1-\epsilon$ for all $\nu$ when $j$ is sufficiently large. The
definition of $q_j$ then implies that
\[
q_j\geq \frac{|\tilde\alpha_{(j)}|}{|\alpha_{(j)}|}\left( 1 -\epsilon
\right) -1,\] and hence $
\frac{q_j|\alpha_{(j)}|}{|\tilde\alpha_{(j)}|} \ \geq \ 1 -\epsilon -
\frac{|\alpha_{(j)}|}{|\tilde\alpha_{(j)}|} \ \longrightarrow
\ 1-\epsilon$ as $j\to\infty$.  On the other hand,
$\frac{q_j|\alpha_{(j)}|}{|\tilde\alpha_{(j)}|}\leq 1$ for all $j$.
Since $\epsilon$ is arbitrary, \eqref{eqn:3.2b} follows.

Let $c:=\max\{\gamma_{\nu}: \nu\in\{1,\dots,m\},(\gamma_1,\dots,\gamma_m)\in\calF\}$, and let $s_j$ be the largest non-negative integer such that 
\begin{equation*}
s_j(\alpha_{(j)\nu}+c)\leq q_j\alpha_{(j)\nu} \quad\hbox{for all }\nu=1,\dots,m.
\end{equation*} 
   Using this, there exists $\tilde r_{(j)}\in\ZZ_{\geq 0}^m$ such that  
\begin{eqnarray*}
Y(\tilde\alpha_{(j)}) &=& Y(q_j\alpha_{(j)} +r_{(j)}) 
 \ = \ Y(s_j\alpha_{(j)}+s_j\gamma_{(j)}+\tilde r_{(j)}),
\end{eqnarray*}
where $\gamma_{(j)}\in\calF$ satisfies
$Y(2\alpha_{(j)}+\gamma_{(j)})\leq Y(\alpha_{(j)})^2$.  It is easy to
see that $\frac{|q_j|}{|s_j|}\to 1$, and hence \eqref{eqn:3.2b} holds
with $q_j,r_{(j)}$ replaced by $s_j,\tilde r_{(j)}$.  Finally,
\begin{eqnarray*}
Y(\tilde\alpha_{(j)})^{\frac{1}{|\tilde\alpha_{(j)}|}}  &=&
Y(s_j\alpha_{(j)}+s_j\gamma_{(j)}+\tilde
r_{(j)})^{\frac{1}{|\tilde\alpha_{(j)}|}} \\ 
&\leq& \bigl(Y(\alpha_{(j)})^{s_j}Y(\tilde
r_{(j)})\bigr)^{\frac{1}{|\tilde\alpha_{(j)}|}} \leq
\bigl(Y(\alpha_{(j)})^{\frac{1}{|\alpha_{(j)}|}}\bigr)^{\frac{s_j|\alpha_{(j)}|}{|\tilde\alpha_{(j)}|}}
C^{\frac{1}{|\tilde\alpha_{(j)}|}}r^{\frac{|\tilde
    r_{(j)}|}{|\tilde\alpha_{(j)}|}}, 
\end{eqnarray*}
where $C,r$ are as in Definition \ref{def:Y}.  Taking the limit as
$j\to\infty$ of the first and last expressions yields $L_2\leq L_1$.
This completes the proof.
\end{proof}


Recall that a positive real-valued function $f$ on a convex set
$C\subseteq\RR^n$ is said to be logarithmically convex if
$f((1-t)a+tb)\leq f(a)^{1-t}f(b)^t$ for all $a,b\in C$; equivalently,
$\log(f)$ is convex.

\begin{lemma} 
\label{lem:3.8}
The function $\theta\mapsto T(\theta)$, defined as in the previous
lemma, is uniformly bounded and logarithmically convex on
$\Sigma_m^{\circ}$ (and hence continuous).
\end{lemma}

\begin{proof}
Boundedness follows easily from subexponential growth: if
$Y(\alpha)\leq Cr^{|\alpha|}$ for all $\alpha\in\ZZ_{\geq 0}^m$ then
$T(\theta)\leq r$ for all $\theta\in\Sigma_m^{\circ}$.

To prove logarithmic convexity, fix
$\theta,\tilde\theta\in\Sigma_m^{\circ}$ and $t\in(0,1)$.  Let
$\alpha_{(j)},\alpha_{(j)}$ satisfy
$\frac{\alpha_{(j)}}{|\alpha_{(j)}|}\to\theta,
\frac{\tilde\alpha_{(j)}}{|\tilde\alpha_{(j)}|}\to\tilde\theta$
as $j\to\infty$ and $|\alpha_{(j)}|=|\tilde\alpha_{(j)}|=:a_j$ for
each $j$.  Let $q_j,\tilde q_j$ be positive integers such that
$\frac{q_j}{q_j+\tilde q_j}\to t$ as $j\to\infty$.

For each $j$ there exist
$\beta_{({j})},\gamma_{(j)},\tilde\gamma_{(j)}\in\calF$ such that
\begin{eqnarray*}
&& Y\bigl(q_j\alpha_{(j)} +\tilde q_j\tilde\alpha_{(j)}
  +\beta_{(j)}+(q_j-1)\gamma_{(j)}+(\tilde
  q_j-1)\tilde\gamma_{(j)}\bigr)  \\ 
&&\quad \ \leq
  \ Y\bigl(q_j\alpha_{(j)}+(q_j-1)\gamma_{(j)}\bigr)\,Y\bigl(\tilde
  q_j\tilde\alpha_{(j)}+(\tilde q_j-1)\tilde\gamma_{(j)}\bigr) \ \leq
  \ Y(\alpha_{(j)})^{q_j}Y(\tilde\alpha_{(j)})^{\tilde q_j}.  
\end{eqnarray*}
Let $\zeta_{(j)}:=q_j\alpha_{(j)} +\tilde q_j\tilde\alpha_{(j)}
+\beta_{(j)}+(q_j-1)\gamma_{(j)}+(\tilde q_j-1)\tilde\gamma_{(j)}$.
Since $\calF$ is bounded, it is easy to see that
$\frac{|\zeta_{(j)}|}{|q_j\alpha_{(j)} +\tilde
  q_j\tilde\alpha_{(j)}|}\to 1$ as $j\to\infty$ and  
\[
\lim_{j\to\infty} \frac{\zeta_{(j)}}{|\zeta_{(j)}|} =
\lim_{j\to\infty} \frac{q_j\alpha_{(j)} +\tilde
  q_j\tilde\alpha_{(j)}}{|q_j\alpha_{(j)} +\tilde
  q_j\tilde\alpha_{(j)}|}
=\lim_{j\to\infty}\frac{q_j\alpha_{(j)}}{(q_j+\tilde q_j)a_j} +
\frac{\tilde q_j\tilde\alpha_{(j)}}{(q_j+\tilde q_j)a_j}=t\theta
+(1-t)\tilde\theta.
\]
Hence
\begin{eqnarray*}
T(t\theta+(1-t)\tilde\theta) &=&
\lim_{j\to\infty}Y(\zeta_{(j)})^{\frac{1}{|\zeta_{(j)}|}} \\ 
& = &  \lim_{j\to\infty}Y(\zeta_{(j)})^{\frac{1}{|q_j\alpha_{(j)}
    +\tilde q_j\tilde\alpha_{(j)}|}} \\ 
&\leq&  \lim_{j\to\infty}
\bigl(Y(\alpha_{(j)})^{\frac{1}{|\alpha_{(j)}|}}\bigr)^{\frac{q_j}{q_j+\tilde
    q_j}}\bigl(Y(\tilde\alpha_{(j)})^{\frac{1}{|\tilde\alpha_{(j)}|}}
\bigr)^{\frac{\tilde q_j}{q_j+\tilde q_j}} =
T(\theta)^tT(\tilde\theta)^{1-t}, 
\end{eqnarray*}
which concludes the proof.
\end{proof}

Given $b\in\partial\Sigma_m=\Sigma_m\setminus\Sigma_m^{\circ}$, define
\begin{equation} 
\label{eqn:3.3c}
T^-(b):= \liminf_{|\alpha|\to\infty,\frac{\alpha}{|\alpha|}\to b}
Y(\alpha)^{\frac{1}{|\alpha|}}.
\end{equation}

\begin{lemma}
Let $b\in\partial\Sigma_m$.  Then 
\[
T^-(b) \ = \ \liminf_{\theta\to b,\; \theta\in\Sigma_m^{\circ}} T(\theta) .
\]
\end{lemma}

\begin{proof}
Let $\{\theta_{(j)}\}_{j\geq 1}$ be a sequence of points in $\Sigma_m^0$ with  $\theta_{(j)}\to b$ as $j\to\infty$, and for each $j$ choose
$\alpha_{(j)}$ such that
\[
|\tfrac{\alpha_{(j)}}{|\alpha_{(j)}|}-\theta_{(j)}|<1/j\, , \quad
|Y(\alpha_{(j)})^{\frac{1}{|\alpha_{(j)}|}}-T(\theta_{(j)})|<1/j \, . 
\]
Then $\frac{\alpha_{(j)}}{|\alpha_{(j)}|}\to b$ as $j\to\infty$, so   
\[
T^{-}(b) \leq  \liminf_{j\to\infty}
Y(\alpha_{(j)})^{\frac{1}{|\alpha_{(j)}|}} \leq \liminf_{j\to\infty}
(T(\theta_{(j)})+{1}/{j}) = \liminf_{j\to\infty} T(\theta_{(j)}).
\]
Hence $\displaystyle T^-(b) \leq \liminf_{\theta\to b,\;
  \theta\in\Sigma_m^{\circ}} T(\theta)$ since the sequence
$\theta_{(j)}$ was arbitrary.

It remains to prove the reverse inequality.  Let
$\sigma=(\sigma_1,\dots,\sigma_m)$ satisfy $\sigma_{\nu}>0$ for each
$\nu$; then $\frac{b+\sigma}{1+|\sigma|}\in\Sigma_m^{\circ}$.  We will
show that  
\begin{equation} 
\label{eqn:3.5a}
T(\tfrac{b+\sigma}{1+|\sigma|}) \leq r^{\frac{|\sigma|}{1+|\sigma|}}
T^-(b)^{\frac{1}{1+|\sigma|}}. 
\end{equation}
(Here $r$ is as in Definition \ref{def:Y}.)

Choose sequences $\alpha_{(j)},\ell_{(j)}$ in $\ZZ_{\geq 0}^m$ such that $|\alpha_{(j)}|\to\infty$ and 
\[
\frac{\alpha_{(j)}}{|\alpha_{(j)}|}\to b \hbox{ with }
Y(\alpha_{(j)})^{\frac{1}{|\alpha_{(j)}|}}\to T^-(b),\quad \hbox{and }
\frac{\ell_{(j)}}{|\alpha_{(j)}|}\to\sigma.
\]
Since $Y$ is weakly submultiplicative with subexponential growth,  
\begin{equation} 
\label{eqn:3.6a} Y(\ell_{(j)}+\alpha_{(j)}+\gamma_{(j)})\leq
Y(\ell_{(j)})Y(\alpha_{(j)})\leq Cr^{|\ell_{(j)}|}Y(\alpha_{(j)})
\end{equation}
for appropriate $\gamma_{(j)}\in\calF$.

We compute
$\frac{\ell_{(j)}}{|\alpha_{(j)}+\ell_{(j)}|}\to\frac{\sigma}{1+|\sigma|}$
and $\frac{\alpha_{(j)}}{|\alpha_{(j)}+\ell_{(j)}|} \to
\frac{b}{1+|\sigma|}$ as $j\to\infty$.  Since $\calF$ is bounded we
also have
$\frac{\gamma_{(j)}}{|\ell_{(j)}+\alpha_{(j)}+\gamma_{(j)}|}\to(0,\dots,0)$
and
$\frac{|\ell_{(j)}+\alpha_{(j)}|}{|\ell_{(j)}+\alpha_{(j)}+\gamma_{(j)}|}\to
1$.  The inequality \eqref{eqn:3.6a} then yields \eqref{eqn:3.5a} by a
similar limiting process as detailed in the previous lemmas.  Finally,
using \eqref{eqn:3.5a}, we have
\[
\liminf_{{\theta\to b,\, \theta\in\Sigma_m^{\circ}}} T(\theta) \leq  
\liminf_{\substack{|\sigma|\to 0\\ \sigma_i>0\,\forall i }} 
T(\tfrac{b+\sigma}{1+|\sigma|})\leq \lim_{|\sigma|\to 0} 
r^{\frac{|\sigma|}{1+|\sigma|}}T^-(b)^{\frac{1}{1+|\sigma|}} = T^-(b), 
\]
which is the desired inequality.
\end{proof}

An immediate consequence of Lemma \ref{lem:3.8} and equation
\eqref{eqn:3.5a} is the following.

\begin{corollary} 
\label{cor:3.6a}
Suppose $T(\phi)\neq 0$ for some $\phi\in\Sigma_m^{\circ}$.  Then
$T(\theta)\neq 0$ for all $\theta\in\Sigma_m^{\circ}$ and $T^-(b)\neq 0$  for all
 $b\in\partial\Sigma_m$.  The same
conclusion holds if $T^-(c)\neq 0$ for some
$c\in\partial\Sigma_m$. \qed 
\end{corollary}

\begin{lemma} 
\label{lem:3.10}
Let $Q$ be a compact subset of $\Sigma_m^{\circ}$.  Then 
\[
\limsup_{|\alpha|\to\infty} \bigl\{ |Y(\alpha)^{\frac{1}{|\alpha|}} - T(\theta(\alpha))|:\ \tfrac{\alpha}{|\alpha|}=:\theta(\alpha)\in Q
\bigr\}\ = \ 0.  
\]
If $T$ is as in the previous corollary, then also
\[
\limsup_{|\alpha|\to\infty} \bigl\{ |\log
Y(\alpha)^{\frac{1}{|\alpha|}} - \log
T(\theta(\alpha))|:\ \tfrac{\alpha}{|\alpha|}=:\theta(\alpha)\in Q
\bigr\}\ = \ 0.  
\]
\end{lemma}

\begin{proof}
Let $L$ denote the first limsup, and let $\{\alpha_{(j)}\}$ be a sequence  for which
\[
\lim_{j\to\infty} |Y(\alpha_{(j)})^{\frac{1}{|\alpha_{(j)}|}}  -
T(\theta_{(j)})| = L,
\]
where $\theta_{(j)}=\frac{\alpha_{(j)}}{|\alpha_{(j)}|}$.  We may assume that $\theta_{(j)}\to\theta\in Q$
by passing perhaps to a subsequence.  Then
\[
|Y(\alpha_{(j)})^{\frac{1}{|\alpha_{(j)}|}} - 
T(\theta_{(j)})| \leq |Y(\alpha_{(j)})^{\frac{1}{|\alpha_{(j)}|}}- T(\theta)|+|T(\theta)    -  T(\theta_{(j)})| 
\]
and as $j\to\infty$, the first expression on the right-hand side goes to
zero by Lemma \ref{lem:3.5a} and the second by continuity of $T$ (Lemma \ref{lem:3.8}).  So $L=0$ as required.

If $T$ is as in the previous corollary, then all quantities inside the second limsup are finite.  To prove this second statement, one do a similar argument as above, writing $\log Y(\alpha_{(j)})^{1/|\alpha_{(j)}|}$, $\log T(\theta_{(j)})$, etc. in place of $Y(\alpha_{(j)})^{1/|\alpha_{(j)}|}$, $T(\theta_{(j)})$.
\end{proof}

For a positive integer $s$, let $h_m(s)$ denote the number of elements
in the set $\{\alpha\in\ZZ_{\geq 0}^m: |\alpha|=s\}$; we have
$h_m(s)=\binom{s+m-1}{s}=\frac{(s+m-1)!}{s!(m-1)!}.$ 

\begin{lemma}
We have
\begin{equation} \label{eqn:lem4.8}
\frac{1}{h_m(s)}\sum_{|\alpha|=s}\log Y(\alpha)^{\frac{1}{|\alpha|}}
\ \longrightarrow
\ \frac{1}{\vol(\Sigma_m)}\int_{\Sigma_m^{\circ}}\log T(\theta)\,
d\theta \quad \hbox{as } s\to\infty,
\end{equation}
where on the right-hand side we integrate over $\theta$ with respect
to the usual $m$-dimensional volume on $\RR^m$, with
$\vol(\Sigma_m)=\int_{\Sigma_m}d\theta$.
\end{lemma}

\begin{proof}
By Corollary \ref{cor:3.6a} we have two cases: either $T$ is never zero on $\Sigma_m^{\circ}$ or $T\equiv 0$.  

We consider the first case.  For convenience write $\theta(\alpha)=\frac{\alpha}{|\alpha|}$.  The
set $\Sigma(s):=\{\theta(\alpha) : |\alpha|=s\}$ is a uniformly
distributed grid of points on $\Sigma_m$ such that the discrete
probability measure
$\frac{1}{h_m(s)}\sum_{|\alpha|=s}\delta_{\theta(\alpha)}$ supported
on $\Sigma(s)$ converges weak$^*$ to $\frac{1}{\vol(\Sigma_m)}d\theta$
as $s\to\infty$.  Since $\theta\to T(\theta)$ is a bounded continuous
function on $\Sigma_m^{\circ}$ and $\vol(\partial\Sigma)=0$,
\[
\frac{1}{h_m(s)}\sum_{|\alpha|=s}\log T(\theta(\alpha))
\ \longrightarrow
\ \frac{1}{\vol(\Sigma_m)}\int_{\Sigma_m^{\circ}}\log T(\theta)\,
d\theta \quad \hbox{as } s\to\infty. 
\]
(To see this, note that the formula holds by weak$^*$ convergence when
$\log T(\theta)$ is replaced by $(1-\chi)\log T(\theta)$ with $\chi$
an arbitrary smooth cutoff function supported in a neighborhood of
$\partial\Sigma$; now shrink the support of $\chi$.) 

Hence to prove \eqref{eqn:lem4.8}, it is sufficient to show that
\begin{equation}
\label{eqn:3.7a}
\Bigl(\frac{1}{h_m(s)}\sum_{|\alpha|=s}|\log
Y(\alpha)^{\frac{1}{|\alpha|}} - \log T(\theta(\alpha))|
\Bigr)\longrightarrow 0 \quad \hbox{as } s\to\infty. 
\end{equation}
Fix $\delta>0$ and define the compact set $Q_{\delta} :=
\{\theta=(\theta_1,\dots,\theta_m)\in\Sigma_m^{\circ}:\ 
\theta_{\nu}\geq\delta\ \forall\,\nu \}$.  For a positive integer $s$, let
\[
L_1(s) := \{\alpha=(\alpha_1,\dots,\alpha_m)\in\ZZ_{\geq
  0}^m:\  |\alpha|=s, \ \tfrac{\alpha}{|\alpha|}\in Q_{\delta} \} 
\]
and let $L_2(s):=\ZZ_{\geq 0}^m\setminus L_1(s)$; write
\[
L_2(s) = \bigcup_{\nu=1}^m \{\alpha\in L_2(s):
\tfrac{\alpha_{\nu}}{s}<\delta\} =: \bigcup_{\nu=1}^m L_{2,\nu}(s). 
\]
Using $\alpha_{\nu}<\delta s$ and $\sum_{\eta\neq\nu}\alpha_{\eta}\leq
s$, we can estimate the size of $L_{2,\nu}(s)$ for each $\nu$ as
$|L_{2,\nu}(s)| \leq \delta s\binom{s+m-2}{s}$.  A calculation then
gives 
\[
\frac{|L_{2}(s)|}{h_m(s)} = \sum_{\nu=1}^m \frac{|L_{2,\nu}(s)|}{h_m(s)}
\leq m\cdot \frac{\delta s\binom{s+m-2}{s}}{\binom{s+m-1}{s} } \leq
\delta m^2. 
\]
Hence 
\begin{eqnarray*}
&& \hskip-1cm \frac{1}{h_m(s)}\sum_{|\alpha|=s}|\log
  Y(\alpha)^{\frac{1}{|\alpha|}} - \log T(\theta(\alpha))| \\ 
&=& \frac{1}{h_m(s)}\sum_{\alpha\in L_1(s)} |\log
  Y(\alpha)^{\frac{1}{|\alpha|}} - \log T(\theta(\alpha))| \\ 
&& \hskip4cm +\ \frac{1}{h_m(s)}\sum_{\alpha\in L_2(s)} |\log
  Y(\alpha)^{\frac{1}{|\alpha|}} - \log T(\theta(\alpha))| \\ 
&\leq& \frac{|L_1(s)|}{h_m(s)}\sup\bigl\{|\log
  Y(\alpha)^{\frac{1}{|\alpha|}} - \log T(\theta(\alpha))| :
  |\alpha|=s,\, \theta(\alpha)\in Q_{\delta}\bigr\} \\ 
&& \hskip5cm + \ \frac{|L_2(s)|}{h_m(s)} (C^{\frac{1}{s}}+1)r \\
&\leq & \sup\bigl\{|\log Y(\alpha)^{\frac{1}{|\alpha|}} - \log
  T(\theta(\alpha))|: \  
|\alpha|=s,\, \theta(\alpha)\in Q_{\delta}\bigr\} \ + \ \delta
m^2(C^{\frac{1}{s}}+1)r, 
\end{eqnarray*}
with $C,r$ as in Definition \ref{def:Y}.  By Lemma \ref{lem:3.10} the
sup in the above line goes to zero as $s\to\infty$, so
\[
\limsup_{s\to\infty}\Bigl(\frac{1}{h_m(s)}\sum_{|\alpha|=s}|\log
Y(\alpha)^{\frac{1}{|\alpha|}} - \log T(\theta(\alpha))| \Bigr) \leq
\delta m^2 2r.
\]
Since $\delta>0$ was arbitrary, \eqref{eqn:3.7a} follows.

For the case $T\equiv 0$, we need to show that  the left-hand side of \eqref{eqn:lem4.8} goes to $-\infty$ as $s\to\infty$. Fix a compact subset $Q$ of $\Sigma_m^{\circ}$. The first part of the previous lemma yields 
\[
\limsup \{Y(\alpha)^{1/|\alpha|}:  |\alpha|\to \infty,\ \alpha/|\alpha|\in Q \}=0.\]  Hence given $\epsilon>0$, 
\[
\sup\{ Y(\alpha)^{1/|\alpha|}: |\alpha|>N,\ \alpha/|\alpha|\in Q \}<\epsilon \]   
for sufficiently large $N$.  Using the notation $L_1(s),L_2(s)$ from the proof of the first case (with $Q$ in place of $Q_{\delta}$), we have 
\[
\frac{1}{h_m(s)}\sum_{|\alpha|=s, \frac{\alpha}{|\alpha|}\in Q}\log Y(\alpha)^{\frac{1}{|\alpha|}} \leq \frac{1}{h_m(s)}\sum_{|\alpha|=s, \frac{\alpha}{|\alpha|}\in Q}\log\epsilon = \frac{|L_1(s)|}{h_m(s)}\log\epsilon \leq \log\epsilon
\]
for $s>N$.  Finally, note that $Y(\alpha)^{1/|\alpha|}$ is uniformly bounded above for all $\alpha$ (say by some constant $M$) since $Y$ has subexponential growth. For all $s$,
\[\frac{1}{h_m(s)}\sum_{|\alpha|=s, \frac{\alpha}{|\alpha|}\not\in Q}\log Y(\alpha)^{\frac{1}{|\alpha|}} = \frac{|L_2(s)|}{h_m(s)}M\leq M.
\]
Altogether, $\frac{1}{h_m(s)}\sum_{|\alpha|=s}\log Y(\alpha)^{\frac{1}{|\alpha|}}\leq M+\log\epsilon$ when $s>N$.  Since $\epsilon$ is arbitrary, the left-hand side of \eqref{eqn:lem4.8} goes to $-\infty$ as required.
\end{proof}

\section{Chebyshev constants} 
\label{sec:cheby}

In this section we construct Chebyshev constants on an algebraic
variety $V\subseteq\CC^n$.  Suppose that $V$ satisfies the properties
\eqref{eqn:properties}.  As before, $R := \CC[z_1,\dots,z_m] \subseteq
\CC[V]$ is a Noether normalization, and $\bv_1,\dots,\bv_d$ are the
polynomials of Section \ref{sec:2}. We will write
$\lambda_1,\dots,\lambda_d$ for the interpolating points denoted by
$p_1,\dots,p_d$ earlier, so that we can use the letter `$p$' to denote
polynomials.  We also introduce some additional notation.

\begin{notation} \rm
 Recall that the basis $\calC$ of $\CC[V]$ was
constructed in Definition~\ref{def:basisCprec}, ordered by $\prec$.
Denote by $\{\be_j\}_{j=1}^{\infty}$ the enumeration of $\calC$
according to $\prec$.  For $f=\sum_ja_j\be_j\in\CC[V]$ we write
$\LT_{\prec}(f)=a_k\be_k$ for the leading term, i.e., $a_k\neq 0$ and
$a_j=0$ for all $j>k$.  For $f,g\in\CC[V]$, write $f\prec g$ if
$\LT_{\prec}(f)\prec\LT_{\prec}(g)$.

In what follows, $\alpha$ will always denote a multi-index in
$\ZZ_{\geq 0}^m$, and we write $\alpha=(\alpha',\alpha_m)$ where
$\alpha'\in\ZZ_{\geq 0}^{m-1}$ and $\alpha_m\in\ZZ_{\geq 0}$.  For
convenience, we will also identify $\alpha$ and $\alpha'$ with
$(\alpha_1,\dots,\alpha_m,0,\dots,0)$ and
$(\alpha_1,\dots,\alpha_{m-1},0,\dots,0)$ in $\ZZ_{\geq 0}^n$ when
using multi-index notation (i.e., in expressions such as $z^{\alpha}$).
\end{notation}

\begin{definition}\rm
Let $\alpha\in\ZZ_{\geq 0}^{m}$ be a multi-index. Define for
$i=1,\dots,d$ the collection of polynomials    
\begin{eqnarray*}
\calM_i(\alpha) &:=&  \Bigl\{p(z)\in\CC[V] : p(z)\, = \,
z^{\alpha}\bv_i \, +\, g(z), \ g(z)\prec z^{\alpha}\bv_i \Bigr\}. 
\end{eqnarray*} 

Fix a compact set $K\subseteq V$. We define the function
$Y_{i}:\ZZ_{\geq 0}^m\to\RR_{\geq 0}$ by 
\[
Y_i(\alpha) \  := \  \inf\{\|p\|_K : p\in\calM_i(\alpha)\}.
\]
\end{definition}

For a fixed $i\in\{1,\dots,d\}$, we will write $\displaystyle
\ell_i(z^{\alpha})$ to denote an arbitrary $g\in\CC[V]$ with $g\prec
z^{\alpha}\bv_i$.  An immediate consequence of Lemma \ref{lem:2.4} is
the following. 

\begin{lemma} 
\label{lem:3.3a}
We have $\bv_i^2 = z_m^{t}\bv_i + \ell_i(z_m^t)$ and $\bv_i\bv_j =
\ell_i(z_m^t)$.  Hence if $p\in\calM_i(\alpha)$,
$q\in\calM_i(\tilde\alpha)$, then
$pq\in\calM_i(\alpha+\tilde\alpha+\gamma_m)$, where
$\gamma_m=(0,\dots,0,t,0,\dots,0)$,
where the $t$ is in the $m$-th slot. \qed
\end{lemma}

\begin{corollary} 
\label{cor:4.4}
The function $Y_i$ is weakly submultiplicative with subexponential
growth.  In particular, 
\begin{equation*} 
Y_i(\alpha+\tilde\alpha+\gamma_m)\leq Y_i(\alpha)Y_i(\tilde\alpha),
\quad \alpha,\tilde\alpha\in\ZZ_{\geq 0}^m. 
\end{equation*}
\end{corollary}

\begin{proof}
Fix indices $\alpha,\tilde\alpha\in\ZZ_{\geq 0}^m$.  Choose
$p\in\calM_i(\alpha)$ such that $\|p\|_K=Y_i(\alpha)$ and
$q\in\calM_i(\tilde\alpha)$ such that $\|q\|_K=Y_i(\tilde\alpha)$.  By
the previous lemma, $pq\in\calM_i(\alpha+\tilde\alpha+\gamma_m)$, so
that $Y_i(\alpha+\tilde\alpha+\gamma_m)\leq
\|pq\|_K\leq\|p\|_K\|q\|_K= Y_i(\alpha)Y_i(\tilde\alpha)$.

Choose $r$ such that $K\subseteq B(0,r)=\{z\in\CC^n : |z|\leq
r\}$. Then $Y_i(\alpha) \leq r^{|\alpha|}\|\bv_i\|_K$, so $Y_i$ has
subexponential growth. 
\end{proof}

As a consequence of the results in the previous section, we have the
following.

\begin{proposition} 
\label{prop:4.5a}
The limit 
\[
T(K,\lambda_i,\theta):=\lim_{\substack{|\alpha|\to\infty \\ \alpha/|\alpha|\to\theta}}Y_i(\alpha)^{\frac{1}{|\alpha|}}  
\]
exists for each $\theta\in\Sigma_m^{\circ}$, and $\theta\mapsto
T(K,\lambda_i,\theta)$ defines a logarithmically homogeneous function
on $\Sigma_m^{\circ}$.  Moreover, we have the convergence
\[
\hskip28pt\frac{1}{h_m(s)}\sum_{|\alpha|=s}\log Y_i(\alpha)^{\frac{1}{|\alpha|}}
\ \longrightarrow
\ \frac{1}{\vol(\Sigma_m)}\int_{\Sigma_m^{\circ}} \log T(K,\lambda_i,\theta)\,
d\theta \quad \hbox{as } s\to\infty.
\]  \qed 
\end{proposition}

\begin{definition}\rm
We call $T(K,\lambda_i,\theta)$ the \emph{directional Chebyshev
constant of $K$ associated to $\lambda_i$ and $\theta$}.   

We call
\[
T(K,\lambda_i):= \exp\left( \frac{1}{\vol(\Sigma_m)}
\int_{\Sigma_m^{\circ}} \log T(K,\lambda_i,\theta)\, d\theta  \right)
\]
the \emph{principal Chebyshev constant of $K$ associated to $\lambda_i$}.

As in \eqref{eqn:3.3c}, we also define $\displaystyle
T^-(K,\lambda_i,b):=\liminf_{|\alpha|\to\infty,\,
  \frac{\alpha}{|\alpha|}\to b} Y(\alpha)^{\frac{1}{|\alpha|}}$  for
$b\in\partial\Sigma_m$. 
\end{definition}

In the proof of the main theorem on transfinite diameter, we will need
to account for polynomials whose leading terms in $\calC$ are of the
form $(*)$.  For $\alpha'\in\ZZ_{\geq 0}^{m-1}$ define 
\[
\tilde\calM(\alpha') := \{ p\in\CC[V]:
\LT_{\prec}(p)=z^{\alpha'}z_m^lz^{\beta} \hbox{ with }  
z_m^lz^{\beta}\in\calB \}. 
\]  
Recall that this means that $l+|\beta|<t$.  Set $\tilde Y(\alpha'):=
\inf\bigl\{\|p\|_K :\ p\in\tilde\calM(\alpha')\bigr\}$.  If
$K\subseteq B(0,r)$ it is easy to see that 
\begin{equation} 
\label{eqn:5.2aa} \tilde Y(\alpha')\leq r^{|\alpha'|}. 
\end{equation}

Also, set $\tilde
T(\alpha'):= \inf\bigl\{\|p\|_K^{\frac{1}{\deg
    p}}:\ p\in\tilde\calM(\alpha')\bigr\}$ and define the function
\begin{equation*}
\tilde T^-(K,\theta')  :=  \liminf_{{|\alpha'|\to\infty, \frac{\alpha'}{|\alpha'|}\to\theta'}  } \tilde T(\alpha') 
\end{equation*}
on $\Sigma_{m-1} := \{\theta'=(\theta_1,\dots,\theta_{m-1}) \in
\RR^{m-1} : \sum_k\theta_k=1\}$.  
 We want to get a lower estimate
for this quantity.  First we make the following observation.  Since
the monomial $z_m^{t-|\beta|}z^{\beta}$ is not in $\calB$ it must be
expressed in $\CC[V]$ with respect to the basis $\calC$ as
\begin{equation} 
\label{eqn:3.2} 
z_m^{t-|\beta|}z^{\beta}  =  \sum_{i=1}^d C_{\beta i}\bv_i + q(z)
\end{equation}
where $\deg q\leq t$, $\LT_{\prec}(q)\prec \bv_1$, and not all
$C_{\beta i}$ are zero.

\begin{lemma} 
\label{lem:4.7a}
Suppose $C_{l\beta i}\neq 0$ for some $i\in\{1,\dots,d\}$.  Then for
each $\theta'\in\Sigma_{m-1}$ we have
\begin{equation} 
\label{eqn:3.3a}
T^-(K,\lambda_i,\theta) \leq \tilde T^-(K,\theta'), 
\end{equation}
where $\theta=(\theta',0) =
(\theta_1,\dots,\theta_{m-1},0)\in\partial\Sigma_m$.
\end{lemma}

\begin{proof}
Fix $\theta'\in\Sigma_{m-1}$ and let $\epsilon>0$.  Let $\{\alpha'_{(j)}\}$ be a
sequence in $\ZZ_{\geq 0}^{m-1}$ with $|\alpha'_{(j)}|\to\infty$,
$\frac{\alpha'_{(j)}}{|\alpha'_{(j)}|}\to\theta'$, and $\tilde
T(\alpha'_{(j)})\to\tilde T^-(K,\theta')$ as $j\to\infty$.

Next, choose a sequence of polynomials $\{p_j\}\subseteq\CC[V]$ such
that $p_j\in\tilde\calM(\alpha'_{(j)})$ and $\|p\|_K^{{1}/{\deg
    p_j}}\leq\tilde T(\alpha'_{(j)})+\epsilon$. Since $\calB$ is
finite, we can assume, by passing perhaps to a subsequence, that
$\LT_{\prec}(p_j)=z^{\alpha'_{(j)}}z_m^lz^{\beta}$ where $l$ and
$\beta$ are the same for all $j$.

Let $Q:=C_{\beta i}^{-1}z_m^{t-l-|\beta|}\bv_i$ and define
$\{P_j\}\subseteq\CC[V]$ by $P_j:=Qp_j$ for each $j$.  Then a
calculation using equation \eqref{eqn:3.2} and Lemma \ref{lem:3.3a}
shows that $P_j\in\calM_i(\alpha_{(j)})$ where
$\alpha_{(j)}=(\alpha'_{(j)},t-|\beta|)$.  Clearly
$\frac{\alpha_{(j)}}{|\alpha_{(j)}|}\to\theta$ as $j\to\infty$ since
$l$ and $|\beta|$ are bounded from above by $t$.  Now
\[
Y_i(\alpha_{(j)})^{\frac{1}{|\alpha_{(j)}|}}\leq
\|Q\|_K^{\frac{1}{|\alpha_{(j)}|}}\|p_j\|_K^{\frac{1}{|\alpha_{(j)}|}}
\leq \|Q\|_K^{\frac{1}{|\alpha_{(j)}|}}\Bigl( \tilde T(\alpha'_{(j)})
+ \epsilon   \Bigr)^{\frac{\deg p_j}{|\alpha_{(j)}|}}.
\]
We take the lim inf as $j\to\infty$.  We have
$T^-(K,\lambda_i,\theta)\leq \tilde T^-(K,\theta')+\epsilon$ since
$\frac{\deg p_j}{|\alpha_{(j)}|}\to 1$, and \eqref{eqn:3.3a} follows
since $\epsilon$ was arbitrary. \end{proof}
  
\begin{corollary}
We have 
\[
\liminf_{|\alpha'|\to\infty}\tilde Y(\alpha')^{\frac{1}{|\alpha'|}}
\ = \  \liminf_{|\alpha'|\to\infty} \tilde T(\alpha) \ \geq
\ \min\bigl\{T^-(K,\lambda_i,\theta) :  i\in\{1,\dots,d\},\,
\theta\in\partial\Sigma_m\bigr\}. 
\]
\qed
\end{corollary}

\section{The transfinite diameter} 
\label{sec:6}

Recall that $\{\be_j\}_{j=1}^{\infty}$ denotes the enumeration of the
basis $\calC$ according to the ordering $\prec$.  For a finite set
$\{\zeta_1,\dots,\zeta_s\}\subseteq V$, define
\begin{equation}
\label{eqn:6}
\Van_{\calC}(\zeta_1,\dots,\zeta_s) \ := \  \det\begin{pmatrix}
1 & 1 & \cdots & 1 \\
\be_2(\zeta_1) & \be_2(\zeta_2) & \cdots & \be_2(\zeta_s) \\
\vdots & \vdots & \ddots & \vdots \\
\be_s(\zeta_1) & \be_s(\zeta_2) & \cdots & \be_s(\zeta_s)
\end{pmatrix}.
\end{equation}
   \def\bff{\mathbf{f}}
   
 As in the previous section, fix a compact set $K\subseteq V$.  We have
 $K\subseteq B(0,r)=\{|z|<r\}$ for some $r>0$.

\begin{notation}\label{not:l_sm_s} \rm For a positive integer $s$,  
\[
V_s:= \sup\{|\Van_{\calC}(\zeta_1,\dots,\zeta_s)| :
\{\zeta_1,\dots,\zeta_s\}\subseteq K \}.
\]
Also, given any positive integer $s$, let $h_s$ denote the dimension of $\CC[V]_{=s}$, let $m_s:= \sum_{\nu=0}^s
h_{\nu}$ denote the dimension of $\CC[V]_{\leq s}$, and let $l_s :=
\sum_{\nu=0}^s \nu h_{\nu}$ denote the sum of the degrees of the basis
elements $\calC\cap\CC[V]_{\leq s}$.
\end{notation}

We now state our main theorem.

\begin{theorem} 
\label{thm:5.2a}
The limit $\displaystyle d(K)=\lim_{s\to\infty} V_{m_s}^{1/l_s}$
exists and we have the formula
\[
d(K) =  \biggl(\prod_{i=1}^d T(K,\lambda_i)\biggr)^{1/d}.
\]
\end{theorem}

To prove the theorem we will need some lemmas.  Recall that $\calB$ is
the collection of monomials given by \eqref{eqn:calB}.


\begin{lemma}
\label{lem:5.3a}
Let $s$ be a positive integer. If $\be_s=z^{\alpha}\bv_i$ for some
integer $i\in\{1,\dots,d\}$ then 
\begin{equation}
\label{eqn:4.2a}
Y_i(\alpha) \leq \frac{V_s}{V_{s-1}} \leq sY_i(\alpha).
\end{equation}
If $\be_s=z^{\alpha'}z_m^lz^{\beta}$ with
$z_m^lz^{\beta}\in\calB\cap\calC$, then 
\begin{equation}
\label{eqn:4.3a}
\tilde Y(\alpha') \leq \frac{V_s}{V_{s-1}} \leq s\tilde Y_i(\alpha').
\end{equation}
\end{lemma}

\begin{proof}
Choose points $\zeta_1,\dots,\zeta_{s-1}$ in $K$ such that
$\Van_{\calC}(\zeta_1,\dots,\zeta_{s-1})=V_{i-1}$.  It is easy to see
that the polynomial
$P(z) :=
\frac{\Van_{\calC}(\zeta_1,\dots,\zeta_{s-1},z)}
{\Van_{\calC}(\zeta_1,\dots,\zeta_{s-1})}$ 
is in $\calM(\alpha)$ by expanding the determinant, and hence
\[
Y_i(\alpha) \leq \|P\|_K \leq \frac{V_s}{V_{s-1}},
\]
which gives the first inequality of \eqref{eqn:4.2a}.

Now choose points $\zeta_1,\dots,\zeta_s$ in $K$ such that
$\Van_{\calC}(\zeta_1,\dots,\zeta_{s})=V_i$ and let $t(z)=\be_s +
\sum_{\nu<s}c_{\nu}\be_{\nu}$ be a polynomial in $\calM(\alpha)$ such
that $\|t\|_K=Y_i(\alpha)$.  Then by properties of determinants,  
\begin{eqnarray*}
V_i &=& \left|\det\begin{pmatrix}
1 & 1 & \cdots & 1 \\
\be_2(\zeta_1) & \be_2(\zeta_2) & \cdots & \be_2(\zeta_s) \\
\vdots & \vdots & \ddots & \vdots \\
\be_{s-1}(\zeta_1) & \be_{s-1}(\zeta_2) & \cdots & \be_{s-1}(\zeta_s) \\
t(\zeta_1) & t(\zeta_2) & \cdots & t(\zeta_s)
\end{pmatrix}\right|  \\
&\leq& \sum_{\nu=1}^s |t(\zeta_{\nu})| 
|V(\zeta_1,\dots,\hat\zeta_{\nu},\dots,\zeta_s)| \ \leq
\ \sum_{\nu=1}^s Y_i(\alpha)V_{s-1}  \ = \ sY_i(\alpha)V_{s-1}, 
\end{eqnarray*}
where we expand along the bottom row.  This gives the second
inequality of \eqref{eqn:4.2a}.

The proof of \eqref{eqn:4.3a} is similar, so we omit it.
\end{proof}

We need to keep track of exponents.  Let $t$ be as in Section \ref{sec:3} (see the paragraph following Corollary \ref{cor:3.3}). Fix an integer
$s>t$.  For an element $z^{\alpha}\bv_i$ there are $d$ choices for $i$
and $h_m(s-t)=\binom{s-t+m-1}{m-1}=\frac{(s-t+m-1)!}{(s-t)!(m-1)!}$
choices for $\alpha$ when $|\alpha|=s-t$.  Hence the number of basis
elements of degree $s$ of the form $(**)$ is $dh_m(s-t)$.         

Let $a_s:=h_s-dh_m(s-t)$ be the number of remaining basis elements, of
the form $(*)$, i.e., $z^{\alpha'}az_m^lz^{\beta}$ with
$\alpha'\in\ZZ_{\geq 0}^{m-1}$ and $z_m^lz^{\beta}\in\calB$.  We then
have the estimate $a_s\leq|\calB|\binom{s+m-2}{m-2}$, where $|\calB|$
denotes the size of the set $\calB$.  Hence   
\begin{equation}
\label{eqn:5.4a}
\frac{a_s}{h_s} \ \leq
\  \frac{|\calB|\binom{s+m-2}{m-2}}{d\binom{s-t+m-1}{m-1}}\ \longrightarrow
\ 0 \quad \hbox{as } s\to\infty, \quad \hbox{and so }
\frac{dh_m(s-t)}{h_s}\longrightarrow 1. 
\end{equation}


Let $\tilde T_s:=\inf\{\tilde T(\alpha'):\ s-t\leq|\alpha'|\leq s\}$.
A straightforward corollary of the previous lemma is the following.

\begin{corollary}
For a positive integer $s>t$, we have 
\begin{equation} 
\label{eqn:4.4a}
\tilde T_s^{sa_s} \Bigl(\prod_{|\alpha|=s}\prod_{i=1}^d Y_i(\alpha)
\Bigr)  \, \leq\, \frac{V_{m_s}}{V_{m_{s-1}}} \, \leq \,
\Bigl(\frac{m_s!}{m_{s-1}!}\Bigr)^2r^{sa_s}
\prod_{|\alpha|=s}\prod_{i=1}^d Y_i(\alpha). 
\end{equation} 
\end{corollary}

\begin{proof}
We apply Lemma \ref{lem:5.3a} to the product
$\frac{V_{m_s}}{V_{m_{s-1}}} = \frac{V_{m_s}}{V_{m_{s}-1}}
\frac{V_{m_s-1}}{V_{m_{s}-2}} \cdots \frac{V_{m_{s-1}+1}}{V_{m_{s-1}}}$.
For the upper estimate, we have
\begin{eqnarray*}
\frac{V_{m_s}}{V_{m_{s-1}}} &=& \frac{V_{m_s}}{V_{m_{s}-1}}
\frac{V_{m_s-1}}{V_{m_{s}-2}}  
\cdots \frac{V_{m_{s-1}+1}}{V_{m_{s-1}}} \\
&=& \Bigl(\frac{V_{m_s}}{V_{m_{s}-1}} \cdots
\frac{V_{m_{s-1}+a_s+1}}{V_{m_{s-1}+a_s}}\Bigr) 
\Bigl(\frac{V_{m_{s-1}+a_s}}{V_{m_{s-1}+a_s-1}} \cdots
\frac{V_{m_{s-1}+1}}{V_{m_{s-1}}} \Bigr) \\  
&\leq&  \Bigl(m_sm_{s-1}\cdots (m_{s-1}+a_s+1)
\prod_{|\alpha|=s}\prod_{i=1}^d Y_i(\alpha)\Bigr) \\ 
& & \qquad \times 
	\Bigl( (m_{s-1}+a_s)\cdots (m_{s-1}+1)
        \prod_{\nu=m_{s-1}+1}^{m_{s-1}+a_s} \tilde
        Y(\alpha'(\be_{\nu})) \Bigr), 
\end{eqnarray*}
where in the last two lines the first large parentheses applies
\eqref{eqn:4.2a} to those fractions $V_k/V_{k-1}$ for which
$\be_{\nu}$ is of the form $(**)$ while the second large parentheses
applies \eqref{eqn:4.3a} to  those fractions for which $\be_{\nu}$ is
of the form $(*)$.    We have also written $\alpha'(\be_{\nu})$ to
denote the multi-index $\alpha'\in\ZZ_{\geq 0}^{m-1}$ for which
$\be_{\nu}=z^{\alpha'}z_m^lz^{\beta}$.  We
have  
\begin{eqnarray*}
& & \hskip-2cm \Bigl(m_sm_{s-1}\cdots (m_{s-1}+a_s+1)
  \prod_{|\alpha|=s}\prod_{i=1}^d Y_i(\alpha)\Bigr) \\ 
& &  \quad \times \Bigl( (m_{s-1}+a_s)\cdots (m_{s-1}+1)
  \prod_{\nu=m_{s-1}+1}^{m_{s-1}+a_s} \tilde Y(\alpha'(\be_{\nu}))
  \Bigr) \\ 
	&\leq&   \Bigl(\frac{m_s!}{m_{s-1}!}
  \prod_{|\alpha|=s}\prod_{i=1}^d Y_i(\alpha)\Bigr) 
	\Bigl(\frac{m_s!}{m_{s-1}!}\prod_{\nu=m_{s-1}+1}^{m_{s-1}+a_s} r^s\Bigr)
\end{eqnarray*}
where we use \eqref{eqn:5.2aa} in the last line. This last expression is the upper estimate in \eqref{eqn:4.4a}.  The lower estimate follows similarly, using the fact that
$s-t\leq|\alpha'(\be_{\nu})|\leq s$ for all
$\nu=m_{s-1}+1,\dots,m_{s-1}+a_s$, so that $\tilde
Y(\alpha'(\be_{\nu}))\geq \tilde T_s^s$ for all $\nu$. 
\end{proof}

Similar reasoning as in the paragraphs before the above corollary give
\[
m_s\leq d\binom{s-t+m}{m} + |\calB|\binom{s-t+m-1}{m-1},
\]
and when $s>t$, 
\[
l_s=\sum_{\nu=1}^{s}\nu h_{\nu} \geq \sum_{\nu=t}^s\nu h_{\nu} \geq 
\sum_{\nu=1}^{s-t}\nu h_{\nu+t} \geq \sum_{\nu=1}^{s-t} \nu\cdot
d\binom{\nu+m-1}{m-1} = dm\binom{s-t+m}{m+1}. 
\]
Then $\displaystyle \frac{m_s}{l_s} \leq \frac{m+1}{m(s-t)} +
\frac{|\calB|(m+1)}{d(s-t)(s-t+m)}\,$, in particular $\frac{m_s}{l_s}\to 0$, and
\begin{equation}
\label{eqn:5.6c}
1\leq (m_s!)^{\frac{1}{l_s}} \leq m_s^{\frac{m_s}{l_s}}
\longrightarrow 1 \quad \hbox{as }s\to\infty. 
\end{equation}

Set $T_s(\lambda_i):=\left(\prod_{|\alpha|=s-t}
Y_i(\alpha)\right)^{\frac{1}{sh_s}}$; then \eqref{eqn:4.4a} becomes 
\begin{equation}
\tilde T_s^{sa_s}\prod_{i=1}^d T_s(\lambda_i)^{sh_s} \ \leq
\ \frac{V_{m_s}}{V_{m_{s-1}}}
\ \leq\ r^{sa_s}\Bigl(\frac{m_s!}{m_{s-1}!}\Bigr)^2 \prod_{i=1}^d
T_s(\lambda_i)^{sh_s}. 
\end{equation}
Write
$V_{m_s}=\frac{V_{m_s}}{V_{m_{s-1}}}\cdots\frac{V_{m_{t+1}}}{V_{m_t}}V_{m_t}$.
Then the above calculation yields the following. 

\begin{corollary}
\[
  \prod_{\nu=t+1}^s \left( \tilde T_{\nu}^{\nu a_{\nu}}\prod_{i=1}^d
T_{\nu}(\lambda_i)^{\nu h_{\nu}} \right) V_{m_t}  \leq  V_{m_s} \leq (m_s!)^2
\prod_{\nu=t+1}^s \left( r^{\nu a_{\nu}}\prod_{i=1}^d
T_{\nu}(\lambda_i)^{\nu h_{\nu}} \right) V_{m_t} . 
\] \qed
\end{corollary}

To prove Theorem \ref{thm:5.2a} we take $l_s$-th roots in the above
inequality and show that the upper and lower estimates have the
desired limit as $s\to\infty$.   
\begin{lemma}
As $s\to\infty$, we have 
\begin{equation} 
\label{eqn:5.8a}
(m_s!)^{\frac{2}{l_s}}\lto 1,\quad \frac{\sum_{\nu=t+1}^s \nu
  a_{\nu}}{l_s}\lto 0,\quad \hbox{and } \frac{sh_s}{(s-t)h_m(s-t)}\lto d. 
\end{equation}
\end{lemma}

\begin{proof}
The first limit follows immediately from \eqref{eqn:5.6c}.  Writing
the left-hand side of the second limit as $\frac{\sum_{\nu=t+1}^s \nu a_{\nu}}{\sum_{\nu=1}^s
  \nu h_{\nu}}$, convergence of this limit to zero follows easily from $\frac{a_s}{h_s}\to 0$ (the
first limit in \eqref{eqn:5.4a}).  The third limit (to $d$) follows easily 
from the second limit in \eqref{eqn:5.4a}.
\end{proof}

\begin{proof}[Proof of Theorem \ref{thm:5.2a}]
We first verify that 
\begin{equation} 
\label{eqn:5.9a} 
T_s(\lambda_i)\to T(K,\lambda_i)^{\frac{1}{d}} \quad \hbox{as }
s\to\infty.
\end{equation}  
By Proposition \ref{prop:4.5a}, 
\begin{eqnarray*}
\Bigl( \prod_{|\alpha|=s-t} Y_i(\alpha)   \Bigr)^{\frac{1}{(s-t)h_m(s-t)}} & =& \exp\Bigl( \frac{1}{h_m(s-t)}\sum_{|\alpha|=s-t}\log Y_i(\alpha)^{\frac{1}{|\alpha|}}  \Bigr) \\
 &\lto&  \ T(K,\lambda_i).
\end{eqnarray*}
Together with the third limit of \eqref{eqn:5.8a} and the definition
of $T_s(\lambda_i)$, we get  \eqref{eqn:5.9a}.  In turn, writing $\tilde l_s=\sum_{\nu=t+1}^s \nu h_{\nu}$, this gives the convergence 
\begin{equation*}
\Bigl(\prod_{\nu=t+1}^s T_{\nu}(\lambda_i)^{\nu h_{\nu}}
\Bigr)^{1/\tilde l_s} \ \lto \ T(K,\lambda_i)^{\frac{1}{d}} \quad
\hbox{as } s\to\infty 
\end{equation*}
of weighted geometric means.  Note that $\tilde l_s/l_s\to 1$ as $s\to\infty$, so we may replace $\tilde l_s$-th roots with $l_s$-th roots in what follows.  We have 
\begin{eqnarray*}
&&\hskip-2cm (m_s!)^{\frac{2}{l_s}} \prod_{\nu=t+1}^s \Bigl( r^{\nu
    a_{\nu}}\prod_{i=1}^d T_{\nu}(\lambda_i)^{\nu h_{\nu}}
  \Bigr)^{\frac{1}{l_s}}V_{m_t}^{\frac{1}{l_s}}   \\
    &=& (m_s!)^{\frac{2}{l_s}} r^{\frac{\sum\nu
      a_{\nu}}{l_s}} \prod_{i=1}^d \Bigl(  \prod_{\nu=t+1}^s
  T_{\nu}(\lambda_i)^{\nu h_{\nu}} \Bigr)^{1/l_s} V_{m_t}^{\frac{1}{l_s}}  \ \lto \ \Bigl(\prod_{i=1}^d T(K,\lambda_i)\Bigr)^{\frac{1}{d}} 
\end{eqnarray*}
as $s\to\infty$, which shows that $\limsup_{s\to\infty} V_{m_s}^{1/l_s}\leq
\Bigl(\prod_{i=1}^d T(K,\lambda_i)\Bigr)^{\frac{1}{d}}$. 

If $T(K,\lambda_i)=0$ for some $i$ then the theorem is proved, with
$d(K)=0$.  Otherwise, $T(K,\lambda_i)>0$ for all $i$; using Corollary
\ref{cor:3.6a} it is easy to see that $T^-(K,\lambda_i,b)>0$ for all
$i=1,\dots,d$ and $b\in\partial\Sigma_m$; and since $\partial\Sigma_m$
is compact, there exists $c>0$ such that $T^-(K,\lambda_i,b)\geq c$
for all $i$ and $b$.  By Lemma \ref{lem:4.7a},
\[
\liminf_{s\to\infty}\tilde T_s\geq \liminf_{|\alpha'|\to\infty}\tilde
T(\alpha') \geq \min_{\theta'\in\Sigma_{m-1}}\tilde T^-(\theta') \geq
\min_{i,b}T^-(K,\lambda_i,b) \geq c,
\]
so there is some uniform constant $\epsilon\in(0,c)$ such that
$T_s>\epsilon$ for all $s>t$, which gives
\[
\prod_{\nu=t+1}^s \Bigl( \epsilon^{\nu a_{\nu}}\prod_{i=1}^d
T_{\nu}(\lambda_i)^{\nu h_{\nu}} \Bigr) V_{m_t} \ \leq \ V_{m_s}. 
\]
Now the $l_s$-th root of the left-hand side of the above goes to
$\Bigl(\prod_{i=1}^d T(K,\lambda_i)\Bigr)^{\frac{1}{d}}$ as
$s\to\infty$ by a similar argument as before.  This concludes the
proof. 
\end{proof}

\section{Transfinite diameter using the standard basis} \label{sec:7}

In this section we verify that the transfinite diameter of the previous section may be computed in terms of the standard (grevlex)  basis of monomials in $\CC[V]$.  Recall that the 
basis for normal forms $\CC[z]_I$ (where $I=\bI(V)$) is given by the
collection of monomials  
\[
\{z^{\gamma}: \gamma\in\ZZ_{\geq
  0},\, z^{\gamma}\not\in\langle\lt(I)\rangle\}.
\]
 Writing $\{\tilde\be_j\}_{j=1}^{\infty}$ for the enumeration of these
 monomials according to grevlex, define $\Van(\zeta_1,\dots,\zeta_M)$
 as in the right-hand side of \eqref{eqn:6} for a finite set $\{\zeta_1,\dots,\zeta_M\}
 \subseteq V$, replacing $\be_j$'s with $\tilde\be_j$'s.  Put
\[
W_{m_s}:= \sup\{|\Van(\zeta_1,\dots,\zeta_{m_s})|:
\{\zeta_1,\dots,\zeta_{m_s}\}\subseteq K \}.
\]

Later in this section we will need to consider Vandermonde determinants formed from other graded polynomial bases.  The Vandermonde determinant associated to a basis $\calF$ will be denoted $\Van_{\calF}(\cdot)$.

\def\f{\mathbf{f}}
\begin{lemma} \label{lem:F}
Let $\calF_1=\{\tilde\f_j\}_{j=1}^{\infty}$ and $\calF_2=\{\f_j\}_{j=1}^{\infty}$ be bases of polynomials for $\CC[V]$, enumerated according to a graded ordering, and suppose that for some positive integer $M$, $\tilde\f_{\tau}=\f_{\tau}$ whenever $\tau>M$.  Then there exists a uniform constant $\kappa\neq 0$ such that for any  integer $\tau\geq M$ and finite set $\{\zeta_1,\ldots,\zeta_{\tau}\}$,
\[
\Van_{\calF_1}(\zeta_1,\ldots,\zeta_{\tau}) = \kappa\Van_{\calF_2}(\zeta_1,\ldots,\zeta_{\tau}).
\]
\end{lemma}

\begin{proof}
Fix the set $\{\zeta_1,\ldots,\zeta_{\tau}\}$ where $\tau\geq M$.  Let 
$E_l=[\tilde\f_j(\zeta_k)]_{j,k=1}^l$ and  $F_l=[\f_j(\zeta_k)]_{j,k=1}^l$
denote the Vandermonde matrices at the $l$-th stage for $l=1,\ldots,\tau$.  With this notation, we have $E_M=P_MF_M$, where $P_M$ is the change of basis matrix from $\{\tilde\f_j\}_{j=1}^M$ to $\{\f_j\}_{j=1}^M$ over the linear space spanned by these polynomials.  In particular, $\det P_M\neq 0$.  Taking determinants, 
$\Van_{\calF_1}(\zeta_1,\ldots,\zeta_M) = \det(P_M)\Van_{\calF_2}(\zeta_1,\ldots,\zeta_M)$.

Similarly, write $E_{\tau}=P_{\tau}F_{\tau}$; then $E_{\tau}$ and $F_{\tau}$ are of the form 
\[
E_{\tau} = 
\left[\begin{array}{c}
E_M\, |\ * \\  \hline
E'
\end{array}\right], \quad
F_{\tau} = \left[\begin{array}{c}
F_M\, |\ * \\  \hline
E'
\end{array}\right], 
\]
the last rows (denoted by $E'$) being the same since $\be_l=\f_l$ when $l>M$.  It follows that $P_{\tau}$ must be of the form $P_{\tau} = \left[\begin{array}{c|c} P_M & * \\ \hline 0 & I    \end{array} \right]$ where $I$ denotes the identity matrix, so that $\det P_{\tau}=\det P_M$.  

Taking $\kappa:=\det P_M$, the lemma follows immediately.
\end{proof}	

Recall that the basis $\calC$ of Definition \ref{def:basisCprec} is made up of the \emph{normal forms} of two types of polynomials:\footnote{cf. Remark \ref{rmk:nfbasis}.}
\begin{eqnarray*}
(*) \quad z^{\alpha}z_m^lz^{\beta}: & & \alpha\in\ZZ_{\geq 0}^{m-1}, \ l+|\beta|\leq t-1 \\
(**) \quad z^{\alpha}z_m^l\bv_i: & & \alpha\in\ZZ_{\geq 0}^{m-1}, \ l\geq 0,\ i=1,\ldots,d.
\end{eqnarray*}
When these polynomials are already normal forms, as in the examples of Section \ref{sec:3}, we have the following theorem.

\begin{theorem} \label{thm:68}
Suppose the polynomials $(*)$ and $(**)$ are already in normal form.  Then $\displaystyle \lim_{s\to\infty} W_{m_s}^{1/l_s} = d(K)$.  (Here $l_s,m_s$ are as in Notation \ref{not:l_sm_s}.)
\end{theorem}

The idea is to show that
$({V_{m_s}^{1/l_s}}/{W_{m_s}^{1/l_s}})\to 1$ as $s\to\infty$, where $V_{m_s}$ is as in the notation of the previous section.  To this end, we analyze the Vandermonde determinants that give these quantities in more detail.

\def\calD{\mathcal{D}}
Write 
$$
\bv_j(z) = \sum_{\beta\in\calD} A_{j\beta}z^{\beta}, \qquad j=1,\ldots,d
$$
where $\calD$ is the collection of all basis monomials that appear in the polynomials $\bv_j$ for all  $j=1\ldots,d$.  Choose constants $c,C>0$ such that for any positive integer $k\leq d$, 
\begin{equation}\label{eqn:cC}
c  \leq  |\det A|\leq C
\end{equation}
whenever $A$ is a $k\times k$ nonsingular square matrix obtained by deleting sufficiently many rows and columns of the $d\times |\calD|$ matrix $\bigl[A_{j\beta}\bigr]_{j,\beta}$.\footnote{Since only the absolute value of the determinant appears, the order of the columns (indexed by $\beta$) is not important.}  There are finitely many possible values for $|\det A|$, so we may take the maximum and minimum of these as our constants.

\medskip

We are interested in $|\Van(\zeta_1,\ldots,\zeta_{m_{\tau}})|$ for a finite set $\{\zeta_1,\ldots,\zeta_{m_{\tau}}\}$.  The value is the same for any graded ordering of the monomials of $\CC[V]_{\leq\tau}$, so let us construct yet another graded ordering that will be convenient for calculation.

Fix the usual grevlex ordering on monomials of degree $<t$.  For $\tau\geq t$, and supposing that monomials of degree $<\tau$ have already been ordered, we order the monomials of degree $\tau$ as follows.  First, list the monomials of the form $(*)$ according to the ordering on $\calC$.   We set up some convenient notation before continuing.
\def\calW{\mathcal{W}}
\begin{notation}\rm 
Let $\calW_0$ be the set consisting of the monomial basis of $\CC[V]_{\leq\tau-1}$ together with the monomials of the form $(*)$ of degree $\tau$.  Let $\bW_0$ denote this same set with our ordering imposed.  (With this notation, the matrices given below are uniquely determined.)  Also, $\bW_k$ will have the same meaning when $\calW_k$, $k=1,2,\ldots$ is defined later in the section.
\end{notation}

Having listed the monomials in $\calW_0$, we will use the elements of $(**)$ to order the remaining monomials in $\CC[V]_{\leq\tau}$.    Before we do this, observe that for $\alpha\in\ZZ_{\geq 0}^m$,
\[
z^{\alpha}\bv_j \ = \ \sum_{\beta\in\calD} A_{j\beta}z^{\alpha+\beta},
\]
and since $z^{\alpha}\bv_j$ is a normal form, each of the monomials in the sum on the right-hand side is a basis monomial.  

Returning to the construction of our ordering, let us enumerate the multi-indices  $\alpha\in\ZZ_{\geq 0}^m$ of total degree $\tau-t$ as $\alpha(1),\alpha(2),\ldots$, according to their order of appearance in the elements of the form $(**)$ in $\calC$.  

The polynomials $\{z^{\alpha(1)}\bv_j\}_{j=1}^d$ are linearly independent by Theorem \ref{thm:freemodule}.  This allows us to choose, for each 
$j=1,\ldots,d$,  a term $z^{\beta(j)}$ of $z^{\alpha(1)}\bv_j$ that is not a term of $z^{\alpha(1)}\bv_i$ whenever $i<j$.  We can also arrange that none of these terms be in $\calW_0$ either, since by the construction of $\calC$ in Section \ref{sec:2}, none of the polynomials  $z^{\alpha(1)}\bv_j$ are in the span of $\calW_0$.   The set of monomials defined by 
\[
\calW_1 \  := \ \{z^{\gamma}: \  z^{\gamma}\in\calW_0  \hbox{ or  } z^{\gamma}=z^{\alpha(1)+\beta(j)}        \}
\]
is therefore a linearly independent subset of basis monomials in $\CC[V]_{\leq\tau}$. 

\begin{remark}  \label{rmk:W}\rm
When $k>1$, note that $z^{\alpha(k)}\bv_j$ is not in the span of $\calW_1$. If it were, then all its monomials would be in $\calW_1$, and, irrespective of how one orders the remaining monomials that are not in $\calW_1$,  the change of basis matrix on $\CC[V]_{\leq\tau}$ from $\calC$ to the monomial basis would not have full rank.  This contradicts the fact that a change of basis matrix must be invertible. \end{remark}

  Now, write  
\[
\left[\begin{array}{c}
\bW_0 \\ \hline z^{\alpha(1)}\bv_1 \\ \vdots \\ z^{\alpha(1)}\bv_d \\ \hline \hbox{rest of }\calC   \\ (\deg\leq\tau) 
\end{array}\right] \ = \ 
\left[\begin{array}{c}
\bW_0 \\ \hline \strut
\sum_{\beta}A_{1\beta}z^{\alpha(1)+\beta} \\ \vdots \\
\sum_{\beta}A_{d\beta}z^{\alpha(1)+\beta}  \\ \hline \hbox{rest of }\calC \\ (\deg\leq\tau)
\end{array}\right]  \ = \ 
\left[\begin{array}{c|c|c}
I  & 0 & 0 \\ \hline
* & A_{(1)} & * \\ \hline
0 & 0 & I 
\end{array}\right]   \left[\begin{array}{c}
\bW_0 \\ \hline
z^{\alpha(1)+\beta(1)} \\
\vdots \\
z^{\alpha(1)+\beta(d)} \\ \hline
\hbox{rest of } \calC \\   (\deg\leq\tau)
\end{array} \right]
\]
where the $(j,k)$-th entry in the block $A_{(1)}$ is given by $A_{j\beta}$ with $\beta=\beta(k)$.  (The `$*$' in the blocks adjacent to $A_{(1)}$ also consist of entries of the form $A_{j\beta}$ but do not enter into subsequent calculations.)  Clearly $c\leq \det A_{(1)}\leq C$ as in \eqref{eqn:cC}.

  Let us write this more compactly as
\[ 
\left[\begin{array}{c}
\bW_0 \\   \hline \hbox{rest} \\ \hbox{of } \calC
\end{array}\right] \ = \ 
\left[\begin{array}{c|c|c}
I  & 0 & 0 \\ \hline
* & A_{(1)} & * \\ \hline
0 & 0 & I 
\end{array}\right]   \left[\begin{array}{c}
\bW_1 \\ \hline
\hbox{rest} \\   \hbox{of } \calC
\end{array} \right] .
\]

The ordering of the remaining monomials is done by repeating the same process as above with the polynomials $z^{\alpha(2)},z^{\alpha(3)},\ldots$, in turn, to form  $\calW_2,\calW_3,\ldots,$ etc.  Assuming that $\calW_{\nu-1}$ has already been constructed, consider the polynomials $\{z^{\alpha(\nu)}\bv_j\}_{j=1}^d$.  They are linearly independent, and by similar reasoning as in Remark \ref{rmk:W}, none of them are in the span of $\calW_{\nu-1}$.  Hence they yield $d$ additional basis monomials which, adjoined to $\calW_{\nu-1}$, form the set $\calW_{\nu}$.  We also have an equation of the form
\begin{equation}\label{eqn:W}
\left[\begin{array}{c}
\bW_{\nu-1} \\   \hline \hbox{rest} \\ \hbox{of } \calC
\end{array}\right] \ = \ 
\left[\begin{array}{c|c|c}
I  & 0 & 0 \\ \hline
* & A_{(\nu)} & * \\ \hline
0 & 0 & I 
\end{array}\right]   \left[\begin{array}{c}
\bW_{\nu} \\ \hline
\hbox{rest} \\   \hbox{of } \calC
\end{array} \right] ,
\end{equation}
with $c\leq |\det A_{(\nu)}|\leq C$ as in \eqref{eqn:cC}.  This is the main formula needed for the proposition below.

\begin{example}\rm
For the complexified sphere $\bV(z_1^2+z_2^2+z_3^2-1)$ in $\CC^3$, the elements of degree $\tau$ in the basis $\calC$ are 
\[
z_1^{\tau},\ z_1^{\tau-1}\bv_1,\ z_1^{\tau-1}\bv_2,\ z_1^{\tau-2}z_2\bv_1,\  z_1^{\tau-2}z_2\bv_2, \ \ldots,
\]   
where $\bv_1=\frac{1}{2}(z_2+iz_3)$ and $\bv_2=\frac{1}{2}(z_2-iz_3)$.
Then 
\[ \calW_0=\{\ldots,z_1^{\tau}\}, \ \calW_1=\{\ldots,z_1^{\tau},z_1^{\tau-1}z_2,z_1^{\tau-1}z_3\},\ 
\calW_2=\calW_1\cup\{z_1^{\tau-2}z_2^2,z_1^{\tau-2}z_2z_3\}.
\]  
\end{example}

\medskip

Recall that for a positive integer $\tau\geq t$,  $h_m(\tau-t)$ coincides with the number of multi-indices $\alpha$ for which $z^{\alpha}\bv_j$ is an element in the basis $\calC$ of degree $\tau$, where $j\in\{1,\ldots,d\}$.
Introduce the notation
\[
b_{\tau}:= \sum_{s=t}^{\tau} h_m(s-t). 
\]
  A straightforward calculation shows that 
\begin{equation}\label{eqn:bt}  b_{\tau}/l_{\tau} \to 0 \  \hbox{ as } \tau\to\infty. \end{equation}

\begin{proposition}
For any collection of points $\{\zeta_1,\ldots,\zeta_{m_{\tau}}\}$, with $\tau\geq t$, we have
\[
c^{b_{\tau}}|\Van(\zeta_1,\ldots,\zeta_{m_{\tau}})| \leq |\Van_{\calC}(\zeta_1,\ldots,\zeta_{m_{\tau}})| \leq C^{b_{\tau}}|\Van(\zeta_1,\ldots,\zeta_{m_{\tau}})|
\]
where $c,C$ are as in \eqref{eqn:cC}.
\end{proposition}

\begin{proof}
The proof is by induction on $\tau$.  We concentrate on the upper inequality involving $C$, and note that the same proof works for the lower inequality.   When $\tau=t$, we have
\begin{eqnarray*}
\left[\begin{array}{c}
\hbox{monomials in } (*) \\ \hbox{of } \deg\leq t \\  \hline 
\bv_1 \\ \vdots \\ \bv_d
\end{array}\right]  =  
\left[\begin{array}{c}
\bW_0 \\ \hline 
\sum_{\beta}A_{1\beta}z^{\beta} \\ \vdots \\  \sum_{\beta}A_{d\beta}z^{\beta}
\end{array}\right]  &=&  
\left[ \begin{array}{c} 
I\ | \  0 \\ \hline A
\end{array}\right]
  \begin{bmatrix}
\bW_1
\end{bmatrix} 
\end{eqnarray*}
and note that in this case, $[\bW_1]$ uses \emph{all} monomials of degree $\leq t$.  Forming Vandermonde determinants, we have 
\[
|\Van_{\calC}(\zeta_1,\ldots,\zeta_{m_t})| \ = \ 
\left|\det \left[ \begin{array}{c} 
I\ | \  0 \\ \hline A
\end{array}\right]\Van(\zeta_1,\ldots,\zeta_{m_t}) \right| \leq C|\Van(\zeta_1,\ldots,\zeta_{m_t})|, 
\]
where we apply \eqref{eqn:cC} and the fact that the determinant in the middle term is the determinant of a $d\times d$ minor of $A$.  This proves the base case.

\medskip

Suppose the inequality holds when $\tau$ is replaced by $\tau-1$.   
 For $j=0,\ldots,b_{\tau}$ let us introduce the convenient notation $\Van_j(\zeta_1,\ldots,\zeta_{m_{\tau}})$ for the ``intermediate'' Vandermonde determinants:
\[
\Van_j(\zeta_1,\ldots,\zeta_{m_{\tau}}) \ = \ 
\det\left[\begin{array}{ccc}
\bW_j(\zeta_1) & \cdots & \bW_j(\zeta_{m_{\tau}}) \\  \hline
z^{\alpha(j+1)}\bv_1(\zeta_1) & \cdots & z^{\alpha(j+1)}\bv_1(\zeta_{m_{\tau}}) \\
\vdots & \ddots & \vdots \\
z^{\alpha(b_{\tau})}\bv_1(\zeta_1) & \cdots & z^{\alpha(b_{\tau})}\bv_d(\zeta_{m_{\tau}}) 
\end{array}\right].
\]
In particular, $|\Van_{h_m(\tau-t)}(\zeta_1,\ldots,\zeta_{m_{\tau}})|=|\Van(\zeta_1,\ldots,\zeta_{m_{\tau}})|$.

Using equation \eqref{eqn:W},
\[ 
|\Van_{\nu-1}(\zeta_1,\ldots,\zeta_{m_{\tau}})| = |\det(A_{(\nu)})|\cdot|\Van_{\nu}(\zeta_1,\ldots,\zeta_{m_{\tau}})|
\leq C|\Van_{\nu}(\zeta_1,\ldots,\zeta_{m_{\tau}})|
\]
for all  $\nu=1,\ldots,b_{\tau}$, and hence by repeated application of the above, 
\[
|\Van_0(\zeta_1,\ldots,\zeta_{m_{\tau}})|\leq C^{h_{m}(\tau-t)}|\Van(\zeta_1,\ldots,\zeta_{m_{\tau}})|.
\]

If we define $\kappa$ by the equation $\Van_{\calC}(\zeta_1,\ldots,\zeta_{m_{\tau}-1})=\kappa\Van(\zeta_1,\ldots,\zeta_{m_{\tau}-1})$, then by Lemma \ref{lem:F}, 
\[ \Van_{\calC}(\zeta_1,\ldots,\zeta_{m_{\tau}})=\kappa\Van_0(\zeta_1,\ldots,\zeta_{m_{\tau}}) \]
as both determinants use the same elements $\{\be_{m_{\tau-1}+1},\ldots,\be_{m_{\tau}}\}$ of degree $\tau$.  Also, note that by the inductive hypothesis, we have $|\kappa|\leq C^{b_{\tau-1}}$. 

 Putting everything together, 
\[
\begin{aligned}
|\Van_{\calC}(\zeta_1,\ldots,\zeta_{m_{\tau}})| 
&\leq C^{b_{\tau-1}}|\Van_0(\zeta_1,\ldots,\zeta_{m_{\tau}})| \\
&\leq C^{b_{\tau-1}+h_m(\tau-t)}|\Van(\zeta_1,\ldots,\zeta_{m_{\tau}})|
=C^{b_{\tau}}|\Van(\zeta_1,\ldots,\zeta_{m_{\tau}})|, 
\end{aligned}  
\]
and the induction is complete.
\end{proof}

Theorem \ref{thm:68} is now an easy corollary. 

\begin{proof}[Proof of Theorem \ref{thm:68}]
Let $K\subset V$ be a compact set.  If $W_{m_{\tau}}=0$ for some $\tau$, then (by a similar argument as in Lemma \ref{lem:F}) $W_{m_{s}}=V_{m_{s}}=0$ for all $s\geq\tau$, and the theorem follows.

 Otherwise, suppose $W_{m_{\tau}}>0$ for all $\tau$. It follows easily from the above proposition that 
\begin{equation}\label{eqn:WV}  c^{b_{\tau}}W_{m_{\tau}} \leq V_{m_{\tau}} \leq C^{b_{\tau}}W_{m_{\tau}}. \end{equation}
 Using \eqref{eqn:bt}, we have $c^{b_{\tau}/l_{\tau}},C^{b_{\tau}/l_{\tau}}\to 1$ as $\tau\to\infty$.  Hence dividing by $W_{m_{\tau}}$ and taking $l_{\tau}$-th roots in \eqref{eqn:WV}, we have 
$(V_{m_{\tau}})^{1/l_{\tau}}/(W_{m_{\tau}})^{1/l_{\tau}}\to 1$ as $\tau\to\infty$.  The theorem is proved. 
\end{proof}

We close the section by sketching an argument that shows how to get rid of the assumption that the products $z^{\alpha}z_m^lz^{\beta}$ and $z^{\alpha}z_m^l\bv_j$ used in Theorem \ref{thm:68} are normal forms.  In general, the methods of this section can be used to construct a basis $\calW$ of linearly independent (but not necessarily normal form) monomials on the variety $V$, made up of the terms in these products.  The same proofs also show that  transfinite diameter defined in terms of $\Van_{\calW}(\cdot)$ gives the same value as that defined in terms of $\Van_{\calC}(\cdot)$.

Now all monomials in $\calW$  are of the form 
\[ z^{\alpha}z^{\beta} = z_1^{\alpha_1}\cdots z_m^{\alpha_m} z_{m+1}^{\beta_{m+1}}\cdots z_n^{\beta_n} \]
with $|\beta|\leq t$, since $\deg\bv_i=t$ for all $i$.  Given $z^{\alpha}z^{\beta}$ as above, consider a monomial $z^{\alpha}z^{\tilde\beta}$ with $|\tilde\beta|\leq s$ for some $s\geq t$.  Then for any compact set $K\subset V$ that avoids the coordinate axes in $\CC^n$,\footnote{Further analysis can be carried out at the end to remove this condition on the axes.}  one can find constants $m$ and $M$, such that, upon evaluating these monomials at any point $\zeta\in K$, 
\begin{equation}\label{eqn:mM}
m^s  \leq  \frac{|z^{\alpha}z^{\tilde\beta}(\zeta)|}{|z^{\alpha}z^{\beta}(\zeta)|} \leq M^s .
\end{equation}
(For example, choose an $M>1$ such that $M\geq\frac{\max\{|z|\, :\ z\in K\}}{\min\{|z_i|\, :\ z=(z_1,\ldots,z_n)\in K\})}$. )

\medskip

 All elements of the (grevlex) monomial basis for $\CC[V]$ have their total degree in the variables  $z_{m+1},\ldots,z_n$ uniformly bounded above (say by $s\geq t$), as a consequence of our hypotheses in Section \ref{sec:3} on Noether normalization.  We can therefore compare these basis monomials to those in $\calW$ using \eqref{eqn:mM}.  

For an integer $\tau\geq t$ and collection of points $\{\zeta_1,\ldots,\zeta_{m_{\tau}}\}\subset K$,  it follows that one can estimate the ratio $\frac{|\Van_{\calW}(\zeta_1,\ldots,\zeta_{m_{\tau}})|}{|\Van(\zeta_1,\ldots,\zeta_{m_{\tau}})|}$ with powers of $m$ and $M$, by repeatedly applying \eqref{eqn:mM} to compare rows of the associated Vandermonde matrices.  One can verify that the growth of these powers is strictly smaller, as a function of $\tau$, than the growth of $l_{\tau}$.  Finally, a similar argument as carried out in the above proof (forming an equation similar to \eqref{eqn:WV}, taking $l_{\tau}$-th roots, etc.) shows that transfinite diameter defined in terms of $\Van(\cdot)$  gives the same value as that defined in terms of $\Van_{\calW}(\cdot)$.

\section{Appendix: The monic basis}

In \cite{rumelylauvarley:existence}, Rumely, Lau and Varley construct
the \emph{sectional capacity} of an algebraic variety.  As in our case
above, Zaharjuta's method plays an essential role.  A so-called
\emph{monic basis} is constructed on the variety with good
multiplicative properties, similar to those of the basis $\calC$ from
Definition~\ref{def:basisCprec}.  Using the monic basis, Chebyshev
constants are then defined in terms of normalized polynomial classes,
and products of Chebyshev constants give the sectional capacity.

The monic basis of \cite[\S4]{rumelylauvarley:existence} is defined in
a very general, abstract setting.  For simplicity, let
$X\subseteq\PP^n$ be an irreducible variety of dimension $m$ and
degree $d$ over $\CC$.  As before, homogeneous coordinates in $\PP^n$
are denoted by $z=[z_0:z_1:\cdots:z_n]$. Then $X$ gives the graded
ring $\CC[X] = \CC[z]/\mathbf{I}(X)$.  The monic basis is a vector
space basis of $\CC[X]$ consisting of homogeneous elements
$\eta_\gamma \in \CC[X]_s$.  Here is a brief sketch of how the monic
basis is constructed:
\begin{enumerate}
\item Write $X=X^{(0)}\supseteq X^{(1)}\supseteq
  X^{(2)}\supseteq\cdots\supseteq X^{(m-1)}$, where for $\ell=1,\dots,m-1$
  we have $X^{(\ell)}=\{z\in X^{(\ell-1)} : z_\ell=0\}$.  We assume
  $X^{(\ell)}$ to be an irreducible variety of dimension $m-\ell$, and that
  the curve $ X^{(m-1)}$ intersects $z_0=0$ in distinct smooth points
  of points of $X^{(m-1)}$; say on the set $D=\{q_1,\dots,q_d\}$.
\item Fix a sufficiently large positive integer $j_0$, for which the
  following holds for $j\geq j_0$: 
\begin{enumerate}
\item For each $i=1,\dots,d$ there exists a rational function on
  $X^{(m-1)}$ with a pole of order $j$ at $q_i$ and no other poles.
\item The collection of rational functions on $X^{(m-1)}$ with poles
  of order at most $j$ on $D$ is isomorphic to the collection of
  homogeneous polynomials on $X^{(m-1)}$ of degree $j$.
\end{enumerate}
\item For each $i,j$ as above, choose a rational function $\eta_{i,j}$
  (normalized appropriately) that satisfies part (a) of the previous
  step.  Choose these functions so that the collection
  $\{\eta_{i,j}\}$ is multiplicatively finitely
  generated.\footnote{This will ensure that the monic basis has good
    multiplicative properties, as can be seen in Example
    \ref{ex:sphere2} below.}
\item Use these rational functions to construct, for each $j$, a basis
  for the homogeneous polynomials of degree $j$ on $X^{(m-1)}$.  (Note
  that these are polynomials in the variables
  $z_0,z_m,z_{m+1},\dots,z_n$ only.)
\item Construct a basis for homogeneous polynomials on the spaces
  $X^{(m-2)},\dots$, $X^{(1)}$, $X$ in turn by inductively adjoining
  monomials in the remaining variables.
\end{enumerate}

The properties of the monic basis and a justification of the above
steps is given in \S\S 4 and 5 of \cite{rumelylauvarley:existence}.
See especially \cite[Thm.\ 4.1]{rumelylauvarley:existence}. 

Note in particular that the monic basis gives a basis of $\CC[X]_s$
for every $s$.  This differs from our setting, where $V \subseteq
\CC^n$ is an affine variety with coordinate ring $\CC[V] =
\CC[z_1,\dots,z_n]/I(V)$.  The basis $\calC$ we construct in
Definition~\ref{def:basisCprec} consists of polynomials that restrict to a
basis of $\CC[V]_{\le s}$ for every $s$.  Thus our basis is compatible
with a \emph{filtration}, while the monic basis in
\cite{rumelylauvarley:existence} is compatible with a \emph{grading}. 

We illustrate how the two bases are related by examining the monic
basis for the complexified sphere considered in Example \ref{ex:sphere}. 

\begin{example}\label{ex:sphere2} \rm  Let 
\[
X=\{[z_0:z_1:z_2:z_3]\in\PP^3 : z_1^2+z_2^2+z_3^2=z_0^2\}
\subseteq \PP^3, 
\]
and $\CC[X] = \CC[z]/\langle z_1^2+z_2^2+z_3^2-z_0^2 \rangle$.  Then
$X^{(1)}$ is the quadratic curve given by $z_1=z_2^2+z_3^2-z_0^2=0$ that intersects
$z_0=0$ in $[0:0:1:\pm i]$.

For each $j=1,2,\dotsc$, it is easy to see that 
\[
\eta_{1,j}(z_0,z_2,z_3) := \left(\frac{z_2+iz_3}{2z_0}\right)^j =
\left(\frac{\bv_1}{z_0}\right)^j 
\]
defines a rational function on $X^{(1)}$ with a pole of order $j$ at
$[0:0:1:-i]$ and no other poles.  The function defined by
\[
\eta_{2,j}(z_0,z_2,z_3):=\left(\frac{z_2-iz_3}{2z_0}\right)^j =
\left(\frac{\bv_2}{z_0}\right)^j 
\]
has the same property in relation to $[0:0:1:i]$.  The rational
functions with at most poles of order $j$ at $[0:0:1:\pm i]$ are then
spanned by
\[
\{1,\eta_{1,1},\eta_{2,1},\eta_{1,2},\eta_{2,2},\dots,
\eta_{1,j},\eta_{2,j}\}.  
\]
A multiplicative generating set is $\{1,\eta_{1,1},\eta_{2,1}\}$.  

Clearing denominators (i.e., multiplying by $z_0^j$) gives the
corresponding basis of homogeneous polynomials of degree $j$ on
$X^{(1)}$.  For example, when $j=2$ we obtain the polynomials
\[
z_0^2,z_0\bv_1,z_0\bv_2,\bv_1^2,\bv_2^2.
\]
To get the basis for the variety $X$, we adjoin powers of $z_1$ to
basis elements for $X^{(1)}$ using the decomposition $\CC[X]_j =
z_1\CC[X]_{j-1} \oplus \CC[X^{(1)}]_j$.  When $j=2$, for example, we
compute that
\begin{equation}
\label{eqn:sphere2}
\begin{aligned}
\CC[X]_2 &= z_1\CC[X]_1 \oplus \CC[X^{(1)}]_2 \\
 &= z_1( z_1\CC[X]_0 \oplus \CC[X^{(1)}]_1) \oplus \CC[X^{(1)}]_2 \\ 
 &= z_1^2\CC[X]_0 \oplus z_1\CC[X^{(1)}]_1 \oplus \CC[X^{(1)}]_2 \\ 
 &= z_1^2\span\{1\} \oplus z_1\span\{z_0,\bv_1,\bv_2\} \oplus
\span\{z_0^2,z_0\bv_1,z_0\bv_2,\bv_1^2,\bv_2^2\} \\  
 &=
\span\{z_0^2,z_0z_1,z_1^2,z_0\bv_1,z_1\bv_1,\bv_1^2,z_0\bv_2,
z_1\bv_2,\bv_2^2\}.  
\end{aligned} 
\end{equation}
The last line gives the monic basis for $j=2$, where the 
basis elements are listed according to the ordering used in
\cite{rumelylauvarley:existence}.   

For arbitrary $j$, monic basis elements $\CC[X]_j$ are either
monomials in $z_0$ and $z_1$ of degree $j$, or are homogeneous
polynomials of the form $z_0^{\alpha_0}z_1^{\alpha_1}\bv_i^{\alpha_2}$
with $\alpha_0+\alpha_1+\alpha_2=j$.  Monomials in $z_0,z_1$ are
listed first in lexicographic order (with $z_0$ preceding $z_1$),
followed by elements of the form
$z_0^{\alpha_0}z_1^{\alpha_1}\bv_i^{\alpha_2}$.  The latter are listed
in increasing order on $i$, then lexicographically by
$\alpha=(\alpha_0,\alpha_1,\alpha_2)\in\ZZ_{\geq 0}^3$.  This
completes the construction of the monic basis for $X$.  
\end{example}

The monic basis constructed in Example~\ref{ex:sphere2} involves
arbitrarily large powers of $\bv_1$ and $\bv_2$.  This is related to
the multiplicative properties of the monic basis described in
\cite[Thm.\ 4.1]{rumelylauvarley:existence}.

It is interesting to compare the monic basis of
Example~\ref{ex:sphere2} to the basis constructed in
Example~\ref{ex:sphere}.  There, we worked with $V =
\mathbf{V}(z_1^2+z_2^2+z_3^2-1) \subseteq \CC^3$.  Since the Zariski
closure of $V$ is $\overline{V} = X =
\mathbf{V}(z_1^2+z_2^2+z_3^2-z_0^2)\subseteq \PP^3$, homogenization
with respect to $z_0$ induces an isomorphism
\[
\CC[V]_{\le j} \simeq \CC[X]_j
\]
for all $j$.  It follows that the basis of Example~\ref{ex:sphere},
when restricted to elements of degree $\le j$, gives a basis of
$\CC[X]_j$.  However, this basis differs from the monic basis in
degree $j$.  For example, when $j = 2$, homogenizing the basis of
Example~\ref{ex:sphere} in degree $\le 2$ gives the homogeneous
polynomials
\[
z_0^2,z_0z_1,z_0\bv_1,z_0\bv_2,z_1^2,z_1\bv_1,z_1\bv_2,z_2\bv_1,z_2\bv_2.
\]
Comparing this to the last line of \eqref{eqn:sphere2}, we see that in
degree $2$, the monic basis uses $\bv_1^2$ and $\bv_2^2$, while our
basis uses $z_1\bv_1$ and $z_1\bv_2$.  These are related by
\[
\bv_1^2 = z_1\bv_1 + {\textstyle\frac14 z_1^2-\frac14 z_0^2}, \quad
\bv_2^2 = z_1\bv_2 + {\textstyle\frac14 z_1^2-\frac14 z_0^2}.
\]

At the conceptual level, the basis $\calC$ constructed in
Definition~\ref{def:basisCprec} focuses on the \emph{module}
properties of the basis, as highlighted in
Theorem~\ref{thm:freemodule}.  In contrast, the monic basis
constructed in \cite{rumelylauvarley:existence} focuses on the
\emph{multiplicative} properties of the basis.  In our treatment, the
multiplicative properties of $\calC$ follow from Lemma~\ref{lem:2.4}.
Our construction is more direct (we avoid the inductive approach needed in
\cite{rumelylauvarley:existence}) but less general than that of
\cite{rumelylauvarley:existence}.

\bibliographystyle{abbrv}
\bibliography{myreferences}

\end{document}